\documentclass[12pt]{article}
\usepackage[margin=1in]{geometry}

\usepackage{amsmath,amsthm,amssymb,amsfonts}
\usepackage[]{algorithm,algorithmic} 
\usepackage{mathtools}
\usepackage{tcolorbox}
\usepackage{bbm}
\usepackage{todonotes}
\usepackage{tensor}
\usepackage{caption}
\usepackage{subcaption}
\usepackage{graphicx}

\newcommand{\real}{\mathbb{R}}

\newcommand{\RR}{\mathbb{R}}
\newcommand{\symMat}{\mathbb{S}}

\usepackage{accents}

\mathchardef\mhyphen="2D 
\DeclareMathOperator*{\argmin}{\arg\!\min}

\DeclarePairedDelimiter{\norm}{\lVert}{\rVert}
\DeclarePairedDelimiter{\abs}{\lvert}{\rvert}

\DeclareMathOperator{\rank}{rank}
\DeclareMathOperator{\nullspace}{\mathbf{nullspace}}
\DeclareMathOperator{\diag}{\mathbf{diag}}

\newtheorem{theorem}{Theorem}
\numberwithin{theorem}{section}

\newtheorem{lemma}[theorem]{Lemma}

\usepackage{url}
\usepackage{titlesec} 
\titleformat{\subsubsection}[runin]{\normalfont\bfseries}{\thesubsubsection.}{10pt}{}

\newcommand{\xsol}{X_\star}
\newcommand{\ysol}{y_\star}
\newcommand{\zsol}{Z_\star}
\newcommand{\rsol}{r_\star}
\newcommand{\pval}{p_\star}
\newcommand{\dval}{d_\star}
\newcommand{\Amap}{\mathcal{A}}
\newcommand{\Ajmap}{\Amap^*}
\newcommand{\xsolset}{\mathcal{X}_\star}
\newcommand{\ysolset}{\mathcal{Y}_\star}
\newcommand{\dm}{n}
\newcommand{\ncons}{m}

\newcommand{\inprod}[2]{\langle #1, #2 \rangle}
\newcommand{\twonorm}[1]{\left\|#1\right\|}
\newcommand{\fronorm}[1]{\left\|#1\right\|_{\mbox{\tiny{F}}}}
\newcommand{\opnorm}[1]{\left\|#1\right\|_{\mbox{\tiny{\textup{op}}}}}
\newcommand{\nucnorm}[1]{\left\|#1\right\|_*}

\DeclareMathOperator{\dist}{dist}
\DeclareMathOperator{\disttwonorm}{dist}
\newcommand{\tr}{\mathop{\bf tr}}

\newcommand{\rpast}{r_{\text{p}}} 
\newcommand{\rcurrent}{r_{\text{c}}}
\newcommand{\bigO}{\mathcal{O}}
\newcommand{\shortTimes}{{\mkern-1mu\times\mkern-1mu}}

\newcommand{\faceplus}[1]{\mathcal{C}_{#1}^{+}}%
\newcommand{\myparagraph}[1]{\textbf{#1}.}

\let\originalleft\left
\let\originalright\right
\renewcommand{\left}{\mathopen{}\mathclose\bgroup\originalleft}
\renewcommand{\right}{\aftergroup\egroup\originalright}

\newcommand{\ben}[1]{{\color{red} (Ben: #1)}}
\renewcommand{\ben}[1]{}

\usepackage{multirow}

\newcommand{\trux}{X^\natural}
\newcommand{\ones}{\mathbf{1}}

\usepackage[utf8]{inputenc}
\usepackage[english]{babel}
\usepackage{xcolor}
\newcommand{\newcontent}[1]{{\color{black}#1}}

\usepackage[algo2e, boxruled, linesnumbered]{algorithm2e}
\SetKwIF{If}{ElseIf}{Else}{If}{}{Else if}{Else}{}
\SetKwFor{Step}{Step $k$: $(k \geq 0)$}{}{}
\SetKwFor{ParFor}{ParallelFor}{}{}
\SetKwFor{For}{For}{}{}

\begin{document}
\title{Revisiting Spectral Bundle Methods: \\
	Primal-dual (Sub)linear Convergence Rates 
	\thanks{\textbf{Funding:} L. Ding was supported by the National Science Foundation CRII 
	award 1657420, grant 1704828, and CCF-2023166. B. Grimmer was supported by the National Science Foundation Graduate Research Fellowship under Grant No. DGE-1650441.}

\author{Lijun Ding\thanks{Wisconsin Institute for Discovery, University of Wisconsin--Madison, Madison, WI, 53705; \texttt{https://lijunding.net}
		} 
	\and Benjamin Grimmer\thanks{Department of Applied Mathematics \& Statistics, Johns Hopkins University, Baltimore, MD 21218;		\texttt{https://www.ams.jhu.edu/~grimmer/}
	}
}
}
	\date{}
	\maketitle
	\begin{abstract}
		The spectral bundle method proposed by Helmberg and Rendl~\cite{helmberg2000spectral} is well established for solving large-scale semidefinite programs (SDP) thanks to its low per iteration computational complexity and strong practical performance.
		In this paper, we revisit this classic method showing it achieves sublinear convergence \emph{rates} in terms of both primal and dual SDPs under merely strong duality, complementing previous guarantees on primal-dual convergence.
		Moreover, we show the method speeds up to linear convergence if (1) structurally, the SDP admits strict complementarity, and (2) algorithmically, the bundle method captures the rank of the optimal solutions. Such complementary and low rank structure is prevalent in many modern and classical applications. The linear convergence result is established via an eigenvalue approximation lemma which might be of independent interest. Numerically, we confirm our theoretical findings that the spectral bundle method, for modern and classical applications, speeds up under these conditions. Finally, we show that the spectral bundle method combined with a recent matrix sketching technique is able to solve an SDP with billions of decision variables in a matter of minutes.
	\end{abstract}

	\section{Introduction}
We consider the problem of solving semidefinite programs (SDPs) of the form
\begin{equation}\label{p} \tag{P}
\begin{aligned}
& \underset{X\in\symMat^{\dm}\subset \RR^{\dm \times \dm}}{\text{maximize}}
& & \langle -C, X \rangle \\
& \text{subject to}
& & \Amap X = b \\
&&& X \succeq 0,
\end{aligned}
\end{equation}
where the decision variable $X \in \symMat^{\dm}\subset \RR^{\dm \times \dm}$ is a symmetric matrix and $n$ may be large (numerically, we consider up to $n\approx 160,000$, resulting in billions of entries in $X$), and the problem data is comprised of a symmetric cost matrix $C\in \symMat^{\dm}$, a linear map $\Amap: \symMat^{\dm} \rightarrow \RR^{\ncons}$, and a right hand side vector $b \in \RR^{\dm}$.
The task of solving \eqref{p} can often be equivalently approached via its dual problem, optimizing over $y\in\RR^{\dm}$,
\begin{equation}\label{d} \tag{D}
\begin{aligned} 
& \underset{y\in\RR^{\ncons }}{\text{minimize}}
& & \langle -b, y \rangle \\
& \text{subject to}
& & \Amap^*y \preceq C
\end{aligned}
\end{equation}
where $\Amap^*$ denotes the adjoint map of $\Amap$. We denote the solution sets of \eqref{p} and \eqref{d} as $\xsolset$ and $\ysolset$ respectively.

Semidefinite programming occurs at the heart of many important large-scale problems (for example, 
matrix completion~\cite{candes2009exact}, max-cut~\cite{goemans1995improved}, community detection~\cite{bandeira2018random}, and phase 
retrieval~\cite{candes2013phaselift}). A huge branch of literature has been devoted to the problem of efficiently solving SDPs like~\eqref{p} \cite{todd2001semidefinite,nesterov1989self,nesterov1994interior,alizadeh1995interior,burer2003nonlinear,glowinski1975approximation,helmberg2000spectral,boyd2011distributed,friedlander2016low,friedlander2016low,yurtsever2019conditional,renegar2014efficient,ding2019optimal}. We refer the reader 
to \cite{monteiro2003first}, \cite[Section 2]{ding2019optimal}, and \cite[Section 3 and 4]{majumdar2019survey}, surveying this myriad of methods. 

Among these methods, spectral bundle methods, proposed by Helmberg and Rendl~\cite{helmberg2000spectral}, stand out due to their low per iteration complexity and fast practical convergence. These two properties are critical to effectively tackling large-scale SDPs (as a high iteration cost may make computing even a single iteration prohibitively slow). In this work, we derive convergence guarantees for a family of spectral bundle methods and identify further computational benefits in both convergence rates and per iteration costs whenever the optimal solutions possess certain low-rank structures, prevalent in many modern applications.

Instead of solving either \eqref{p} or \eqref{d} directly, Helmberg and Rendl's spectral method considers the following equivalent penalization dual problem:
for any sufficiently large $\alpha$, e.g., larger than the trace of any maximizer of \eqref{p} \cite[Lemma 6.1]{ding2019optimal}\footnote{The lemma in \cite{ding2019optimal} as written
	requires the primal solution to be unique, but applies equally when there are multiple solutions, replacing the condition $\alpha > \tr(\xsol)$ by $\alpha >\sup_{\xsol \in \xsolset}\tr(\xsol)$.}, \eqref{d} is equivalent to (in the sense of having the same optimal value and solution set)\footnote{Actually, the 
	spectral bundle method of Helmberg and Rendl requires the trace of every feasible $X$ for \eqref{p} to be the same and 
	deals with the eigenvalues instead of the maximum of the eigenvalues and zero in \eqref{eq: penaltySDP}.
	However, the method extends directly to the general setting without fixed trace.}
\begin{equation}\tag{\texttt{pen}-D}
\begin{aligned}\label{eq: penaltySDP}
& \underset{y\in\RR^{\ncons}}{\text{minimize}}
& & F(y) :\,= \langle -b, y \rangle +\alpha \max\{\lambda_{\max}(\Ajmap y-C), 0\}.
\end{aligned}
\end{equation}
In Section~\ref{sec:def}, we formally define bundle methods and the considered family of spectral variants for solving SDPs. The main idea behind these methods is to approximate the nonnegative eigenvalue function $\alpha \max\{\lambda_{\max}(\Ajmap y-C), 0\}$ by a maximum of lower bounds indexed by a small SDP representable set.
Roughly speaking, the considered family of spectral methods maintain this approximation using $\rpast$ past eigenvectors and $\rcurrent$ current eigenvectors of the matrix $\Ajmap(y)-C$ evaluated at the past and current iterates respectively. We denote a method from this family as $(\rpast,\rcurrent)$-SpecBM. There is a rich history of studying methods of this form, which we provide detailed connections to in Appendix \ref{sec: historicalRemark}.

Importantly, the use of a small SDP representable set makes the problem of minimizing this eigenvalue approximation tractable. This can result in per iteration complexities much lower than that of many ADMM type methods \cite{boyd2011distributed} or the second-order bundle method in \cite{oustry2000second}, which both require full eigenvalue decomposition of an 
$n\times n$ matrix, requiring $\bigO(\dm^3)$ operations in general. Section~\ref{sec: ImportantsubproblemSolver} discusses the iteration cost and computational advantages of this approach (which can rely on as few as one eigenvector computation per iteration, $\rcurrent=1$). 

Spectral bundle methods for solving SDPs have received considerable attention since being first proposed in~\cite{helmberg2000spectral} and have been considered in many extended settings by the algorithmic variants of~\cite{helmberg2002spectral,apkarian2008trust,helmberg2014spectral}. 
Despite the success of these methods, past convergence theory, e.g., \cite[Lemma 5]{helmberg2000spectral}
\footnote{We note that \cite[Lemma 5]{helmberg2000spectral} only shows the dual objective converges and did not address the primal convergence. An analysis of primal convergence is given in \cite[Theorem 15.6]{helmberg2004cutting}. Our results in Section \ref{sec: analysis} can also be used to conclude primal convergence from dual convergence.}, mainly focuses on whether the iterates converge, rather than their convergence rates. This work's analysis of spectral methods aims to explain and predict empirical performance and quantify the tradeoffs related to these methods' low iteration costs.

\myparagraph{Our contributions}
In this work, we establish convergence guarantees for a broad family of spectral bundle methods, $(\rpast,\rcurrent)$-SpecBM, and show that these convergence rates speed up substantially under appropriate structural conditions, matching observed performance.
\begin{itemize}
	\item \textbf{Sublinear Spectral Bundle Method convergence rates}: In Theorem 
	\ref{thm: sublinearates}, we show that any configuration of $(\rpast,\rcurrent)$-SpecBM admits a $\bigO(1/\epsilon^3)$
	convergence rate in terms of the dual objective and $\bigO(1/\epsilon^6)$ in terms of the primal merely assuming strong duality holds. 
	Additionally, under strict complementarity (formalized in Section \ref{sec: analtical conditon}), dual and primal convergence
	speeds up to $\bigO(1/\epsilon)$ and $\bigO(1/\epsilon^2)$ respectively. 
	
	\item \textbf{Linear convergence under low-rankness}: In Theorem \ref{thm: linear convergence of Block SBM under the extra condition strict complementarity}, we further show \emph{linear convergence} if (1) strict 
	complementarity and dual uniqueness hold and (2) the number of eigenvectors computed each iteration $\rcurrent$ exceeds the largest of any primal optimal solution's rank. \newcontent{This fast convergence result is based on a novel eigenvalue approximation Lemma~\ref{lem: importantLemmaQuadraticAccurateModel} showing that when the optimal solution is low rank, the bundle method's model objective becomes quadratically accurate.}

	\item \textbf{Scalability and storage reductions under low-rankness}:
	Finally, we show that spectral bundle methods can scale up to tackle large-scale SDPs whenever solutions possess the appropriate low-rank structure. This is accomplished in part by incorporating the matrix sketching ideas of~\cite{tropp2017practical,yurtsever2017sketchy}. Whenever the primal optimal solutions are low-rank, this tool enables the spectral bundle method to be applied without ever storing a matrix $X$ with $\dm^2$ entries, attaining the notion of storage optimality discussed in~\cite[Section 1.2]{ding2019optimal}. Section~\ref{sec: numerics} demonstrates these scalability gains following from configuring the spectral method based on our linear convergence theory and utilizing the improved time and space complexity 
induced  by sketching.
\end{itemize} 
\subsection{Low Rankness and Algorithm Performance} \label{sec: discussion On low-rankness}
A ubiquitous structure among applications of semidefinite programming is that the solutions of \eqref{p} are \emph{low rank}. For many applications, an explicit upper bound on this rank is available from domain knowledge: 
\begin{itemize}
    \item \textbf{Recommendation systems and matrix completion:} The user-item rating matrix underlying many recommendation systems is usually incomplete and requires filling-in missing entries. This problem is also known as matrix completion. It has been observed in~\cite[figure 5]{chen2018harnessing} and~\cite[figure 3]{fan2019factor} that for different movie-lens datasets~\cite{harper2015movielens}, the underlying complete matrix has rank no more than 30 even though there are thousands of users and items. 
    \item \textbf{Sensor networks and Euclidean distance matrix completion:} In sensor networks, usually only a few pairs' distances are known or measured while the distance matrix for all the pairs is desired. This is also known as the Euclidean distance matrix completion problem. When distances are actually measured in our three-dimensional world, the resulting distance matrix must have rank at most three~\cite{alfakih1999solving,so2007theory}. 
    \item \textbf{Community detection and $\mathbb{Z}_2$ synchronization:} The problem of community detection aims to identify clusters in a graph where nodes within the same community are more likely to have an edge. The SDP formulation of this problem~\cite{guedon2016community,li2021convex} has a solution with rank no more than the number of clusters. In an idealized two cluster problem and its continuous version~\cite{abbe2014decoding,li2021convex} ($\mathbb{Z}_2$ synchronization), the optimal SDP solution has rank one. Both of these models play an important role in understanding the theoretical limit of computational methods~\cite{bandeira2017tightness,abbe2014decoding}.
\end{itemize}
Other examples of the prevalence of low-rank optimal solutions include Max-Cut, which has solution rank no more than 30 for various datasets \cite[Table 1]{ding2020regularity}, and phase retrieval \cite{candes2015phase}, which always has a unique rank one solution.

The presence of low-rank solutions is critical to enabling our linear convergence guarantees and improvements in iteration cost (in both time and space complexity) for spectral bundle methods. Our linear convergence results require the parameter $r_c$ to be greater or equal to the dimension of the null space of every dual optimal solution's slack matrix. Under strict complementarity and primal-dual uniqueness, this condition is equivalent to $r_c\geq \rank(\xsol)$, i.e, the parameter $r_c$ is larger than or equal to the primal solution rank. Both conditions are satisfied for many applications as verified in \cite{ding2020regularity}. Moreover, we observe in our numerics that even if the dual uniqueness condition fails (which indeed occurs for matrix completion~\cite{ding2020regularity}) as long as strict complementarity holds, we only need $r_c \geq \rank(\xsol)$. 

A larger choice of $r_c$ increases the per iteration computational complexity of the spectral bundle method (discussed in Section \ref{sec: ImportantsubproblemSolver}). Hence selecting $r_c$ near $\rank(\xsol)$ maintains fast linear convergence while notably reducing the method's per iteration time complexity for applications with low-rank solutions. Similarly, we discuss storage reduction techniques in Section \ref{subsec:sketching} utilizing an estimate upper bounding the optimal solution rank. As a result, the amount of needed memory can scale linearly with this estimate, establishing improved space complexity for low-rank applications as well.

\subsection{Paper Organization and Notation} Section \ref{sec:def} formally introduces bundle methods (based on proximal regularization and aggregation) and the family of spectral bundle methods considered, $(\rpast,\rcurrent)$-SpecBM. Then Section \ref{sec: analysis} presents our main convergence theory. Section \ref{sec: numerics} numerically demonstrates convergence speed-ups whenever the parameter $\rcurrent$ is chosen larger than the optimal solution rank, matching our theory, and shows how matrix sketching ideas can be applied to notably scale up this approach (to problems with billions of entries in $X$).

\myparagraph{Notation} We denote members of the optimal solution sets by $\xsol \in \xsolset$ and $\ysol \in \ysolset$. We equip 
$\symMat^{\dm}$ and $\real^{\ncons}$ with the trace inner product and the dot product respectively, and denoted both as $\inprod{\cdot}{\cdot}$. The induced norms are both denoted as $\norm{\cdot}$. For a symmetric matrix $A\in \symMat^\dm$, we denote its eigenvalues as 
$\lambda_{\max}(A)=\lambda_{1}(A)\geq \dots \geq \lambda_{\dm}(A)$ with a corresponding set of orthonormal eigenvectors $v_1,v_2,\dots,v_n$. The notation $\mathbb{S}_+^n \subset \mathbb{S}^n$ denotes the set of $n\times n$ symmetric positive semidefinite matrices.
The matrix operator two norm, Frobenius norm, and nuclear norm are denoted as $\opnorm{\cdot}$, $\fronorm{\cdot}$, and $\nucnorm{\cdot}$ respectively. We denote the maximum nuclear norm of the primal solution set by $D_{\xsolset} =\sup_{\xsol\in\xsolset} \nucnorm{\xsol}$
and of a penalized dual level set by $D_{y_0} = \sup_{F(y)\leq F(y_0)} \norm{y}.$
The dual slack matrix for each $y\in \RR^{\ncons}$ is defined as $Z(y)\mbox{:=} C-\Amap^*y$. 
The operator norm of $\Amap^*$ is defined as $\opnorm{\Amap^*}=\max_{y\in \RR^{\ncons},\twonorm{y}\leq 1}\fronorm{\Amap^*y}$. For a closed set $\mathcal{X}\subset \RR^{\ncons}$ and a point $z\in \RR^{\ncons}$, we define the distance of $z$ to  $\mathcal{X}\subset \RR^{\ncons}$ as $\dist(z,\mathcal{X})=\inf_{x\in \mathcal{X}}\norm{x-z}$. 

	\section{Preliminaries and Spectral Bundle Methods} \label{sec:def}

In this section, we first review two standard conditions (strong duality and strict complementarity) of well-behaved semidefinite programs. Then Section~\ref{subsec:bundle-methods} introduces the framework for proximal bundle methods and Section~\ref{subsec:spectral-bundle} specializes this to define spectral bundle methods by utilizing carefully constructed eigenvalue approximations (based on a bundle of past and current eigenvectors). Finally, in Section~\ref{subsec:computation}, we discuss the needed computations and per iteration costs to implement such a spectral bundle method.

Throughout, we assume that the pair of semidefinite programming problems~\eqref{p} and~\eqref{d} satisfy \emph{strong duality}: namely that the solution sets $\xsolset$ and $\ysolset$ are nonempty, 
compact and each pair $(\xsol,\ysol)\in \xsolset\times \ysolset$ has zero duality 
gap
\[
\pval := \inprod{-C}{\xsol} = \inprod{-b}{\ysol}=:\dval.
\]
Note that we require $\xsolset$ and $\ysolset$ to be nonempty and compact instead of just  $\pval =\dval$. This
condition holds whenever Slater's conditions are satisfied by both \eqref{p} and \eqref{d} and the map $\Amap$ is surjective.

Following~\cite[Definition 4]{alizadeh1997complementarity}, we say a pair $(\xsol,\ysol)$ with dual slack matrix $\zsol(\ysol) = C-\Ajmap(\ysol)$ satisfies \emph{strict complementarity} if 
\[
\rank(\xsol)+\rank(\zsol)=\dm. 
\]
Whenever such a pair exists, we say~\eqref{p} and~\eqref{d} satisfy 
\emph{strict complementarity}. This condition is satisfied by generic SDPs~\cite{alizadeh1997complementarity} as well as by many well structured SDPs~\cite{ding2020regularity}.

\subsection{Proximal Bundle Methods}\label{subsec:bundle-methods}
A bundle method for solving a generic minimization problem $\min_{y\in\RR^{ \ncons}} f(y)$ constructs an approximation of the objective $\bar f_t$ at each iteration $t$, typically based on (sub)gradient evaluations of $f$ at a sequence of points $z_t$ (utilizing both the past and current iterates). We denote the set of subgradients of a convex function $f$ by 
$\partial f(y) = \{g\in\RR^\ncons \mid f(y') \geq f(y)+\langle g, y'-y\rangle  \text{ for all } y' \}$, referred to as the subdifferential of $f$ at $y$.

Each iteration of a proximal bundle method computes the following proximal step minimizing this model of the true objective
\begin{equation}\label{eq:prox-step}
 z_{t+1} \in \argmin \bar f_t(y) + \frac{\rho}{2}\|y-y_t\|^2
\end{equation}
where $y_t \in \real^{m}$ is the current reference point (proximal center) and $\rho>0$. The point $z_{t+1}\in \real^{m}$ serves two purposes: (i) it is used to construct the next model objective function $\bar f_{t+1}$ and (ii) if $z_{t+1}$ offers sufficient descent, defined for some fixed $\beta\in(0,1)$ as
$ f(z_{t+1}) \leq f(y_t) - \beta\left(f(y_t) - \bar f_t(z_{t+1})\right)$,
then the next iteration takes $y_{t+1}=z_{t+1}$ (called a descent step), otherwise the proximal center is not changed $y_{t+1}=y_t$ (called a null step). This process is formalizes in Algorithm~\ref{alg:bundle}.

For the sake of simplifying our development and to take advantage of existing convergence theory for proximal bundle methods, we will assume that this model $\bar f_t$ is constructed satisfying the following three properties: $\bar f_{t+1}$ is a lower bound on the true objective $f$
\begin{equation} \label{eq:cond1}
    \bar f_{t+1}(y) \leq f(y) \qquad \text{ for all }y\in\RR^{ \ncons}\ ,
\end{equation}
$\bar f_{t+1}$ is lower bounded by the linearization given by some subgradient $g_{t+1}\in \partial f(z_{t+1})$ computed after each~\eqref{eq:prox-step}
\begin{equation} \label{eq:cond2}
   \bar f_{t+1}(y) \geq f(z_{t+1}) + \langle g_{t+1}, y - z_{t+1}\rangle  \qquad \text{ for all }y\in\RR^{ \ncons}\ , 
\end{equation}
$\bar f_{t+1}$ is lower bounded by the linearization given by the subgradient $s_{k+1}:= \rho(z_{t+1}-y_t)\in \partial \bar f_{t+1}(z_{t+1})$ (i.e., the subgradient certifying that $z_{t+1}$ minimizes~\eqref{eq:prox-step})
\begin{equation} \label{eq:cond3}
   \bar f_{t+1}(y) \geq \bar f_t(z_{t+1}) + \langle s_{t+1}, y - z_{t+1}\rangle  \qquad \text{ for all }y\in\RR^{ \ncons}\ , 
\end{equation}

A bundle method with a full memory may construct $\bar f_t$ as the maximum of all $f(z_{\tau}) + \langle g_{\tau}, \cdot - z_{\tau}\rangle$ with $\tau\leq t$. Alternatively, a bundle method with cut aggregation may utilize a much simpler model given by the maximum of the two required lower bounds~\eqref{eq:cond2} and~\eqref{eq:cond3}, where the second bound serves as an aggregation of all the previous subgradient bounds. Spectral bundle methods construct a more specialized model approximating the eigenvalue function in~\eqref{eq: penaltySDP}.
\begin{algorithm2e}[t]
  \SetAlgoNoLine
	\KwData{$z_0=y_0 \in \RR^n$, $\bar f_0=f(y_0)+\langle g_0, \cdot -y_0\rangle$}
	\Step{}{
	Compute candidate iterate $ \displaystyle z_{t+1} \leftarrow \argmin_{z \in \RR^{d}} \bar f_t(z) + \frac{\rho}{2}\|z-y_t\|^2$\;
	\If(\tcp*[f]{Descent step}){$\beta(f(y_t) - \bar f_t(z_{t+1})) \leq  f(y_t)-f(z_{t+1})$}{
	  Set $y_{t+1} \leftarrow z_{t+1}$\;
	}
	\uElse(\tcp*[f]{Null step}){
	  Set $y_{t+1} \leftarrow y_{t}$\;}
	  Compute $\bar f_{t+1}$ without violating~\eqref{eq:cond1}, \eqref{eq:cond2}, or \eqref{eq:cond3}\tcp*[r]{Update Model}
	}
	\caption{Proximal Bundle Method}
	\label{alg:bundle}
\end{algorithm2e}

\subsubsection{Proximal Bundle Method Convergence Guarantees}\label{subsec:bundle-rates}
Under any method of constructing models $\bar f_{t+1}$ satisfying these conditions, \eqref{eq:cond1}-\eqref{eq:cond3}, the proximal bundle method is known to converge to a minimizer for any closed convex objective that attains its minimum value. Here we briefly review the existing guarantees on this method's rates of convergence. In our analysis, we will utilize these results as a blackbox to bound the penalized dual formulation's objective gap $F(y_t)-F(\ysol)\leq \epsilon$ when specialized to spectral bundle methods.

In particular, we are interested in guarantees on the sequence of proximal centers $y_t$, which by definition have non-increasing function value (only changing at descent steps). \newcontent{Moreover, Algorithm~\ref{alg:bundle} has bounded $M=\sup_{t\geq 0}\{\|g_t\|\}<\infty$ and $D=\sup_{t\geq 0}\{\dist(y_t, \ysolset)\}<\infty$ since both sequences of iterates, $z_t$ and $y_t$, produced by the bundle method are well known to converge whenever $\ysolset\neq\emptyset$~\cite[(7.64)]{ruszczynski2006nonlinear}.}

Several previous works have bounded the total number of steps (descent and null) required to reach a target optimality gap $\epsilon>0$. The earliest such guarantee for the proximal bundle method was given by Kiwiel~\cite{kiwiel2000efficiency}, showing that after $\bigO(1/\epsilon^3)$ total steps (descent and null), the method has $f(y_t)-f(\ysol)\leq \epsilon$. More recently, Du and Ruszczynski~\cite{Du2017} showed under a quadratic growth bound (like that given by Lemma~\ref{lem: qg}), this convergence rate improves to $\bigO(\log(1/\epsilon)/\epsilon)$.
Recently, Diaz and Grimmer~\cite{DG2020} derived slightly more general versions of these $\bigO(1/\epsilon^3)$ and $\bigO(1/\epsilon)$ bounds (improving the latter by a log factor). \newcontent{Following Theorems 2.1 and 2.3 of~\cite{DG2020}, the following convergence guarantees hold.
\begin{theorem}\label{thm:prox-bm-convergence}
    For any convex $f$ with nonempty set of minimizers, the iterates $y_t$ of Algorithm~\ref{alg:bundle} have $f(y_t)-\inf f\leq \epsilon$ for all
    $$ t\geq \bigO\left(\frac{\rho M^2D^4}{\beta(1-\beta)^2\epsilon^3}\right) \ . $$
    Additionally, if some $\mu>0$ has $f(y) - \inf f\geq \mu \dist(y,\ysolset)^2$ for all $y$, this bound improves to
    $$ t\geq \bigO\left(\frac{M^2}{\beta(1-\beta)^2\min\{\mu/\rho, 1\}\epsilon}\right)\ . $$
\end{theorem}
The big-$\bigO$ notation above suppresses universal constants as well as additive terms with a lesser order of dependence on $1/\epsilon$.}

\subsection{Spectral Bundle Methods}\label{subsec:spectral-bundle}
Directly applying the above proximal bundle method to the dual penalized formulation~\eqref{eq: penaltySDP} requires computing a subgradient of the maximum eigenvalue function. For each $z_t\in\mathbb{R}^m$, a subgradient is given by
$ -b + \alpha \Amap vv^\top \in \partial F(z_t)$
where $v$ is any top eigenvector of $\Amap^*z_t-C$ if $\lambda_{\max}(\Amap^*z_t-C)>0$ and is $0$ otherwise. This corresponds to the affine lower bound
\begin{equation}\label{eq: single-eigenvector-lb}
     F(y) \geq \langle-b, y\rangle + \langle \alpha vv^\top, \Amap^*y-C\rangle \ . 
\end{equation}

The key idea behind spectral bundle methods is to improve on this lower bound by utilizing \emph{infinitely} many affine lower bounds to model the objective more closely. For any $Z\in \symMat^\dm$:
\[
\max \{\lambda_{\max}(Z),0\} = \max_{\inprod{X}{I}\leq 1,X\succeq 0}\inprod{X}{Z}.
\]
Hence we may rewrite the penalized dual objective $F$ as 
\begin{equation}\label{eqn: dualpenaltyobjectiverewrite}
F(y) = \max_{\inprod{X}{I}\leq \alpha , X\succeq 0} \inprod{-b}{y} +\inprod{X}{\Amap^*y-C}.
\end{equation}
Of course, this form is no easier to solve than the original penalized form \eqref{eq: penaltySDP}. However, this perspective gives a natural way to model $F$ by replacing the constraint set $\{\inprod{X}{I}\leq \alpha, X\succeq 0\}$ with a smaller convex set. One choice is that we compute a matrix $V\in \RR^{n\times r}$ for some small value $r$ with orthonormal columns, i.e., $V^\top V= I\in \RR^{r\times r}$.
Then we form the following spectral lower bound model based on $V$:
\begin{equation}\label{eqn: modelFoverV}
\bar{F}_{V}(y) :\,= \max_{\inprod{S}{I}\leq \alpha , S\in \symMat_+^r} \inprod{-b}{y} +\inprod{VSV^\top}{\Amap^*y-C}.
\end{equation}
When $r=1$ and $V$ is selected as a top eigenvector of $\Amap^*z_t-C$, $\bar{F}_{V}$ recovers~\eqref{eq: single-eigenvector-lb}. For $r>1$, selecting $V$ spanning $v$ gives a strictly better approximation.
The choice of $V$ should consist of eigenvectors based on the present iterate and (potentially) the accumulation of past spectral information.

To ensure this spectral model always satisfies the needed aggregate bound~\eqref{eq:cond3}, one further refinement is needed. In addition to the selection of past and current eigenvectors generating $V$, the spectral bundle method maintains a carefully selected weighted sum of the past spectral bounds as $\bar{X}\in \symMat_+^\dm$ with
$\nucnorm{\bar{X}}=\tr(\bar{X})\leq \alpha $.
Then we build our model using this matrix $\bar{X}$ along with $V\in \RR^{\dm \times r}$ as: 
\begin{equation}\label{eq: modelaggreagte}
\bar{F}_{(V,\bar{X})}(y) :\,= \max_{\eta \tr(\bar X) +\tr(S)\leq \alpha , \eta \geq 0, S\in \symMat_+^r} \inprod{-b}{y} +\inprod{\eta\bar{X}+VSV^\top}{\Amap^*y-C}.
\end{equation}
In the following subsection, we detail the exact method for selecting $r_c$ top eigenvalues from the current iteration and $r_p$ past eigenvalues to construct $V$ and the selection of $\bar X$ needed to ensure~\eqref{eq:cond3} holds.

\subsubsection{A Family of Spectral Bundle Methods: $(\rpast,\rcurrent)$-SpecBM} \label{sec: SpectralbundleMethodExactDefinition}
The considered family of spectral bundle methods utilizes $\bar r = \rpast +\rcurrent$ eigenvectors in its spectral approximations $V$.
We maintain two sequences of iterates, initialized with $z_0=y_0=0\in \RR^\ncons$ and an initial spectral model $\bar{F}_{(V_0,\bar{X}_0)}$ given by $\bar{X}_0=0$ and any $V_0\in \real^{\dm \times \bar{r}}$ with orthonormal columns.
The sequence of iterates $y_t$ serve as the reference point (proximal center) for the proximal subproblem $\bar{F}_{(V_t,\bar{X}_t)}(z) + \frac{\rho}{2}\|z-y_t\|^2$ to be minimized, producing the next candidate point $z_{t+1}$. $(\rpast,\rcurrent)$-SpecBM does this by iterating the same basic three steps as Algorithm~\ref{alg:bundle}, formalized in Algorithm~\ref{alg:spec-bundle}:

First, each iteration $t$ of $(\rpast,\rcurrent)$-SpecBM computes $z_{t+1} := z^\star_{t}, \eta^\star_{t}$, and $S^\star_{t}$ by solving the proximal subproblem
\begin{equation}\label{eq: subproblemAMSpectralbundleMethod}
	\min_{z} \ \max_{\eta \tr(\bar X_t)  +\tr(S)\leq \alpha , \eta \geq 0,S\in \symMat_+^{\bar{r}}} \inprod{-b}{z} +\inprod{\eta\bar{X}_t+V_t S V_t^\top}{\Amap^*z-C}+\frac{\rho}{2}\norm{z-y_t}^2. 
\end{equation}
In Section~\ref{sec: ImportantsubproblemSolver}, we detail how to compute such a minimax solution. Second, $(\rpast,\rcurrent)$-SpecBM computes the next reference point $y_{t+1}$ using the same descent test of the proximal bundle method. Finally, $(\rpast,\rcurrent)$-SpecBM computes the next spectral model $(V_{t+1},\bar X_{t+1})$ as follows: Define the matrix attaining the inner maximization above as
\begin{equation} \label{eq: X_tdefinition}
	X_{t}= \eta^\star_{t}\bar{X}_t+V_t S^\star _{t}V_t^\top.
\end{equation}
Denote the eigenvalue decomposition of $S^\star_t= Q_1\Lambda_1 Q_1^\top + Q_2 \Lambda_2 Q_2^\top$, where $\Lambda_1$ consists of the largest 
	$\rpast$ eigenvalues, and $\Lambda_2$ consists of the rest of the eigenvalues. Then set the next model's $\bar{X}_{t+1}$ as
	\begin{equation}\label{eqn: updatebarX_t+1FC}
	\bar{X}_{t+1}= \eta^\star_t \bar{X}_t +V_tQ_2\Lambda_2Q_2^\top V_t^\top. 
    \end{equation}
    \newcontent{The rationale behind this update as suggested in \cite{helmberg2000spectral} is that 
    the important spectral information of $X_t$, apart from $\bar{X_t}$, are 
    those eigenvectors $V_tQ_1$  which correspond to the larger eigenvalues. By explicitly keeping those eigenvectors $V_tQ_1$ in $V_{t+1}$, we hope  
the model accuracy of $\bar{F}$ is improved in the next round. Another choice of aggregation is to set $\bar{X}_{t+1} = X_t$, which can be analyzed by similar reasoning to our analysis. We follow \eqref{eqn: updatebarX_t+1FC} as this is the one proposed in \cite{helmberg2000spectral}.}

Compute  
$\rcurrent$ orthonormal top eigenvectors $v_1,\dots,v_{\rcurrent}$ of the current dual  $\Amap^*z_t^\star -C$. Then set the next model's orthonormal $V_{t+1}$ to spa these $\rcurrent$ current eigenvectors and the $\rpast$ aggregate directions $V_tQ_1$. \newcontent{For example, the next orthonormal matrix $V_{t+1}$ can be computed as a $QR$ factorization of $[V_tQ_1; v_{1},\dots,v_{\rcurrent}]$, setting $V_{t+1}=Q$.}

\begin{algorithm2e}[t]
  \SetAlgoNoLine
	\KwData{$\bar X_0 = 0$, $z_0=y_0=0 \in \RR^n$, orthonormal $V_0\in \real^{\dm \times \bar{r}}$}
	\Step{}{
	Compute candidate $z_{t+1} := z^\star_{t}, \eta^\star_{t}$, and $S^\star_{t}$ solving the subproblem~\eqref{eq: subproblemAMSpectralbundleMethod}\;
	\If(\tcp*[f]{Descent step}){$\beta(F(y_t) - \bar F_{(V_t,\bar X_t)}(z_{t+1})) \leq  F(y_t)-F(z_{t+1})$}{
	  Set $y_{t+1} \leftarrow z_{t+1}$\;
	}
	\uElse(\tcp*[f]{Null step}){
	  Set $y_{t+1} \leftarrow y_{t}$\;}
	  Compute the decomposition $S^\star_t= Q_1\Lambda_1 Q_1^\top + Q_2 \Lambda_2 Q_2^\top$\;
	  Compute $v_1,\dots,v_{\rcurrent}$ top eigenvectors of $\Amap^*z_t^\star -C$\;
	  Set next $\bar X_{t+1}$ by~\eqref{eq: X_tdefinition} and~\eqref{eqn: updatebarX_t+1FC} \tcp*[r]{Update Model}
	  Set next $V_{t+1}$ spanning $v_1,\dots,v_{\rcurrent}$ and $V_tQ_1$\;
	}
	\caption{$(r_p,r_c)$-Spectral Bundle Method}
	\label{alg:spec-bundle}
\end{algorithm2e}

Note selecting of top eigenvectors of $\Amap^*z_t^\star -C$ can be viewed as selecting the primary directions describing infeasibilities in the dual slack matrix $C-\Amap^*z_t^\star$. 

\subsection{Computational Details and Concerns}\label{subsec:computation} 
\label{sec: ImportantsubproblemSolver}
In Section \ref{sec: SublinearRates}, we verify that the necessary inequalities~\eqref{eq:cond1}, \eqref{eq:cond2} and~\eqref{eq:cond3} are all satisfied by $(\rpast,\rcurrent)$-SpecBM's construction of its model $F_{(V_{t+1},\bar X_{t+1})}$. Consequently, the proximal bundle method's $O(1/\epsilon^{3})$ and $O(1/\epsilon)$ (see Section~\ref{subsec:bundle-rates}) objective value convergence guarantees apply.

To efficiently implement the spectral bundle method, one needs to efficiently solve the minimax optimization subproblem~\eqref{eq: subproblemAMSpectralbundleMethod} at each iteration. Define the $t$-th spectral set $\mathcal{W}_t$ as 
\begin{equation}\label{eq: tthspectralset}
\mathcal{W}_t = \{\eta \bar{X}_{t} + V_{t}SV_{t}^\top \mid \eta \geq 0,\;S\in \symMat_+^{\bar{r}},\;\text{and}\;\eta \tr(\bar{X}_t) + \tr(S)\leq \alpha \}.
\end{equation}
Hence the minimax subproblem~\eqref{eq: subproblemAMSpectralbundleMethod} is equivalent to
\begin{equation}\label{eq: subproblemSpectralbundleMethodsolveStep1}
\begin{aligned} 
&\min_{z} \max_{X\in \mathcal{W}_t} \inprod{-b}{z} +\inprod{X}{\Amap^*z-C}+\frac{\rho}{2}\norm{z-y_t}^2\\
= &\max_{X\in \mathcal{W}_t} \min_{z}  \inprod{-b}{z} +\inprod{X}{\Amap^*z-C}+\frac{\rho}{2}\norm{z-y_t}^2
\end{aligned} 
\end{equation}
where the equality follows from Sion's minimax theorem.

By completing the square, we find that the inner minimization is achieved only when $z =y_t + \frac{1}{\rho}\left(b-\Amap X\right)$. Consequently, the subproblem reduces to 
\begin{equation}\label{eq: subproblemSpectralbundleMethodsolveStep2}
\begin{aligned} 
&\max_{X\in \mathcal{W}_t} \min_{z}  \inprod{-b}{z} +\inprod{X}{\Amap^*z-C}+\frac{\rho}{2}\norm{z-y_t}^2\\ 
=&\max_{X\in \mathcal{W}_t} \inprod{-b}{y_t} +\inprod{X}{\Amap^*y_t-C} -\frac{1}{2\rho}\twonorm{b-\Amap X}^2 \\
= &-\min_{X\in \mathcal{W}_t} \inprod{b}{y_t} +\inprod{X}{C-\Amap^*y_t} +\frac{1}{2\rho}\twonorm{b-\Amap X}^2. \\ 
\end{aligned} 
\end{equation}
This last minimization problem in \eqref{eq: subproblemSpectralbundleMethodsolveStep2} is the augmented Lagrangian problem of Problem \eqref{p} with 
the decision variable $X$ restricted to $\mathcal{W}_{t+1}$ instead of $\symMat_+^\dm$. This interpretation as solving an augmented Lagrangian during its iterations has been explored by~\cite[section 5.2]{lemarechal2001lagrangian}.

Recalling the definition of~\eqref{eq: tthspectralset}, this augmented Lagrangian problem in \eqref{eq: subproblemSpectralbundleMethodsolveStep2} is a low dimension subproblem. Namely, it is equivalent to
\begin{equation}\label{eq: subproblemSpectralbundleMethodsolveStep3}
\begin{aligned} 
\min_{(\eta,S)\in \mathcal{S}_t}  
f_t(\eta,S),
\end{aligned} 
\end{equation}
where $f_t(\eta,S) = \inprod{b}{y_t} +\inprod{\eta \bar{X}_t+V_tSV_t^\top}{C-\Amap^*y_t}+\frac{1}{2\rho}\twonorm{b-\Amap\left( \eta \bar{X}_t+V_tSV_t^\top \right)}^2$ and $\mathcal{S}_t = \{(\eta,S)\mid {S\succeq 0,\;\eta \geq 0,\;\tr(S)+\tr(\bar{X}_t)\eta \leq \alpha}\}$.

As $\bar{r} = \rpast+\rcurrent$, this problem has dimension $1+\bar{r}(\bar{r}+1)/2$. This problem could be solved with an accelerated first-order method as gradients of $f_t$ and the projection to the constraint set $\mathcal{S}_t$ (after 
a proper scaling)
can be done with time complexity $\bigO(\bar{r}^3)$ (see 
\cite[Appendix B]{dingSpec2022} for detail). Alternatively, interior 
point method (described in \cite[Section 6]{helmberg2000spectral}) can be applied with $\bigO({\bar{r}}^6)$ time complexity per iteration
due to inverting an ${\bar{r}\choose 2} \times {\bar{r} \choose 2}$ matrix. This is particularly useful when the problem~\eqref{eq: subproblemSpectralbundleMethodsolveStep3} fails to be well conditioned.  Note $\rpast$ and $\rcurrent$ can be chosen as low as $0$ and $1$, respectively, which would yield a dimension two subproblem over 
$(S,\eta)\in \RR^{2}$ with constraints $S\geq 0,\eta\geq 0$ and $S+\tr(\bar{X}_t)\eta \leq \alpha$, and a quadratic objective $f_t$. In this case, explicit formulas for the optimal $S_t^\star$ and $\eta_t^\star$ 
can be derived easily to avoid numerical optimization. Note once optimal $S_t^\star$ and $\eta_t^\star$ are found, the needed subproblem solution is exactly $z_t^\star = y_t +\frac{1}{\rho}\left(b-\Amap(\eta_t^\star \bar{X}_t +V_tS_t^\star V_t^\top)\right)$.

\myparagraph{Storage concerns} We note that just for the purpose of computing $S^\star_t$ and $\eta^\star_t$, one needs not to store $\bar{X}_t$ but only need to store $\Amap \bar{X}_t $,
$\inprod{C}{\bar{X}_t}$, and $\tr(\bar{X}_t)$, as we may write $\bar f_t$ as 
\begin{equation}
\begin{aligned}
\bar f_t(\eta,S) =&\inprod{b}{y_t} +\eta \inprod{ \bar{X}_t }{C}
-\eta \inprod{\Amap(\bar{X})}{y_t}
+\inprod{V_tSV_t^\top }{C-\Amap^*y_t}\\ &+\frac{1}{2\rho}\twonorm{b-\eta(\Amap
	\bar{X}_t)- \Amap(V_tSV_t^\top)}^2.
\end{aligned}
\end{equation}

The updates of $\Amap \bar{X}_t $,
$\inprod{C}{\bar{X}_t}$, and $\tr(\bar{X}_t)$ are also easy given the low rank updates of $\bar{X}_t$ in \eqref{eqn: updatebarX_t+1FC}. Keeping only 
$\Amap \bar{X}_t $,
$\inprod{C}{\bar{X}_t}$ and $\tr(\bar{X}_t)$ is advantageous when $\Amap$ and $\inprod{C}{\cdot}$ can be quickly applied to low rank matrices. Moreover,
one can recover the matrix $\bar{X}_t$ without the need of storing $\bar{X}_t$ for the spectral bundle method
using the matrix sketching idea in \cite{tropp2017practical}. We further illustrate this in Section \ref{subsec:sketching}, showing such techniques enable the spectral bundle method to be applied to far larger problem instances.

	\section{Analysis of $(\rpast,\rcurrent)$-SpecBM}\label{sec: analysis}
In this section, we present our convergence guarantees for the considered family of bundle methods whenever strong duality holds, with improved guarantees whenever strict complementarity holds. Under any selection of the algorithmic parameters, Theorem~\ref{thm: sublinearates} below gives sublinear convergence guarantees for both primal and dual solutions, showing $X_t$ and $y_t$ converge in terms of feasibility and objective gap. Whenever $\rcurrent$ is selected large enough (to capture the rank of the primal optimal solutions), Theorem~\ref{thm: linear convergence of Block SBM under the extra condition strict complementarity} shows much faster linear convergence occurs.

Our sublinear convergence guarantees for $(\rpast,\rcurrent)$-SpecBM apply for any selection of the algorithmic parameters $\rho>0,\beta\in (0,1),\rpast\geq 0,\rcurrent\geq 1$. The only requirement is that the penalization parameter be selected large enough $\alpha \geq 2 D_{\xsolset}$. Under any such parameter selection, the spectral bundle method converges at a rate of $\bigO(1/\epsilon^3)$, which improves to $\bigO(1/\epsilon)$ whenever strict complementarity holds. This is formalized below and proven in Section~\ref{sec: SublinearRates}.
\begin{theorem}\label{thm: sublinearates}
	Suppose strong duality holds. Given any $\beta\in(0,1)$, $\rcurrent\geq 1$, $ \rho>0$,
	$\alpha \geq 2D_{\xsolset}$, $V_0\in \RR^{\dm\times \bar{r}}$, and 
	$z_0=y_0\in \RR^{\ncons}$, and target accuracy $\epsilon\in (0,\frac{1}{2})$,   $(\rpast,\rcurrent)$-SpecBM
	produces a solution pair $X_t$ and $y_t$ with
	$F(y_t)-F(\ysol)\leq \epsilon$ and
	\begin{align*}
	\text{approximate primal feasibility: }& \quad\|b - \Amap X_{t}\|^2 \leq \epsilon, \quad  X_{t}\succeq 0,\\
	\text{approximate dual feasibility: }& \quad \lambda_{\min}(C - \Ajmap y_{t})\geq -\epsilon,\\
	\text{approximate primal-dual optimality: }&\quad  |\langle b, y_t \rangle - \langle C, X_t\rangle| \leq \sqrt{\epsilon}
	\end{align*}
	by some iteration $t\leq \bigO(1/\epsilon^3)$.
	Moreover, if additionally strict complementarity holds, then these conditions are reached by some iteration $t\leq \bigO(1/\epsilon)$.
\end{theorem}
Deriving the convergence rates above relies on leveraging the existing analysis~\cite{Du2017,grimmer2019general,kiwiel2000efficiency,DG2020} for generic proximal bundle methods to specialized spectral models. The recent work~\cite{DG2020} further shows adaptive, nonconstant stepsize selection rules (replacing $\rho$ by a sequence of parameters $\rho_t$) can improve the $\bigO(1/\epsilon^3)$ rate to $\bigO(1/\epsilon^{2})$. Practically implementing such schemes (and computing needed constants) may be difficult and so constructing such an adaptive spectral bundle method is beyond the scope of this work (but may be of future interest).

Even greater improvements in convergence follow if (in addition to strict complementarity) the number of eigenvalues computed at each iteration satisfies 
\begin{equation}\label{eq: zddefinition}
\rcurrent \geq r_d :=\max_{y\in \ysolset}\dim(\nullspace(Z(\ysol)))
\end{equation} 
where $r_d$ denotes the largest dimension of the null space of dual slack matrices. As discussed in Section \ref{sec: discussion On low-rankness}, for several modern applications \cite{candes2009exact,candes2013phaselift,recht2010guaranteed,ding2020regularity} of \eqref{p}, $\xsol$ is unique, admits rank $\rsol :\,=\rank(\xsol)\ll \dm$, and satisfies strict complementarity under certain structural probabilistic assumptions. If in addition, the dual solution is unique, then we only need $\rcurrent \geq r_d =\rsol$.\footnote{Here the equality $r_d =\rsol$ is due to strict complementarity and the rank-nullity theorem.} The requirement $\rcurrent \geq \rsol$ can be motivated from an eigenvalue computational perspective as the bottom $\rsol$ eigenvalues of the slack $Z(y_t)$ start to coalesce once $y_t$ is close to $\ysolset$. Moreover, we numerically observe in Section \ref{sec: numerics} that even if there are multiple dual solutions, $\rcurrent\geq \rsol $ suffices to yield quick convergence while $\rcurrent <\rsol$ induces slow convergence. 

Under these conditions, $(\rpast,\rcurrent)$-SpecBM will converge linearly once $y_t$ is close enough to $\ysolset$ (note the above sublinear convergence guarantees provide a constant bound on the number of iterations required to reach any fixed neighborhood).  
This is formalized below and proven in Section \ref{sec: Proof of Local linear convergence of Block bundle method}.

\begin{theorem}\label{thm: linear convergence of Block SBM under the extra condition strict complementarity}
	Suppose strong duality and strict complementarity holds. 
	Then under proper selection of $\rho$ and any $\beta\in [0,\frac{1}{2}],\alpha \geq 2D_{\xsolset},V_0\in \RR^{\dm \times \bar{r}}, z_0=y_0 \in \RR^{\ncons}$, and $\rcurrent \geq r_d$, after at most
	$T_0$ steps, $(\rpast,\rcurrent)$-SpecBM will subsequently only take descent steps and converge linearly to an optimal solution.
	Consequently, for any $\epsilon\in(0,\frac{1}{2})$, $(\rpast,\rcurrent)$-SpecBM produces a solution pair $X_t$ and $y_t$ with  $F(y_t)-F(\ysol)\leq \epsilon$ and 
	\begin{align*}
	\text{approximate primal feasibility: }& \quad\|b - \Amap X_{t}\|^2 \leq \epsilon, \quad  X_{t}\succeq 0,\\
	\text{approximate dual feasibility: }& \quad \lambda_{\min}(C - \Ajmap y_{t})\geq -\epsilon,\\
	\text{approximate primal-dual optimality: }&\quad  |\langle b, y_t \rangle - \langle C, X_t\rangle| \leq \sqrt{\epsilon}
	\end{align*}
	by some iteration $t\leq T_0 + \bigO(\log(1/\epsilon))$.
\end{theorem}
Bounds on $T_0$ and proper selection of $\rho$ are discussed at the beginning of Section \ref{sec: Proof of Local linear convergence of Block bundle method}.

\subsection{Preliminaries on Growth Bounds and Primal-Dual Convergence}\label{sec: analtical conditon}
Before proving our two main convergence theorems for spectral bundle methods, we develop a few preliminary lemmas. These results characterize the effect of strong duality and strict complementarity on the penalized dual problem~\eqref{eq: penaltySDP} and then relate approximately minimizing~\eqref{eq: penaltySDP} to approximate feasibility and optimality of both~\eqref{p} and~\eqref{d}.

Whenever the considered primal-dual SDP pair satisfies strong duality, they each satisfy a growth bound, ensuring that the objective gap and/or level of infeasibility grow quickly as one moves away from the set of optimal solutions. Under strict complementarity, we find this growth is quadratic, which facilitates our faster convergence rates for spectral bundle methods. 
\begin{lemma}[Quadratic Growth]\cite[Section 4]{sturm2000error} \label{lem: qg}
	Suppose strong duality holds for \eqref{p} and \eqref{d}, then there exists $\zeta_1,\zeta_2\geq 1$, such that for any fixed $\epsilon>0$, there are some $\gamma_1,\gamma_2>0$ such that for all $y$ with $F(y)\leq \inprod{-b}{\ysol}+ \epsilon$, and all $X\succeq 0$ with 
	$|\inprod{C}{X}- \inprod{C}{\xsol}|\leq \epsilon$ and $\twonorm{\Amap{X}-b}\leq \epsilon$:
	\begin{align} 
	\dist^{\zeta_1}(y,\ysolset) & \leq  \frac{1}{\gamma _1}(F(y)-F(\ysol)), \label{eq: dualg} \\ \dist^{\zeta_2}(X,\xsolset) & \leq \frac{1}{\gamma_2}\left(\abs{\inprod{C}{X}-\inprod{C}{\xsol}}+\twonorm{\Amap{X}-b}\right). \label{eq: primalg}
	\end{align}
	If in addition, strict complementarity holds for some pair of primal dual solutions  $(\xsol,\ysol)\in 
	\xsolset\times \ysolset$, then $\zeta_1=\zeta_2=2$.
\end{lemma}
\begin{proof}
	Define the sublevel set $S_1=\{y\mid F(y)\leq \inprod{-b}{\ysol}+ \epsilon\}$, and the set  $S_2=\{X\mid X\succeq 0,
	|\inprod{C}{X}- \inprod{C}{\xsol}|\leq \epsilon,\,\text{and}\, \twonorm{\Amap{X}-b}\leq \epsilon\}$. We first show these two sets are compact. Indeed, 
	using \cite[Theorem 7.21]{ruszczynski2006nonlinear}, the penalty form 
	$\min_{X\succeq 0} g(X):= \inprod{C}{X} +\gamma \twonorm{\Amap X-b}$
	has the same solution set as 
	the primal SDP \eqref{p} for some large $\gamma >0$ . Thus the compactness of the set $S_1$ and $S_2$ is ensured by the compactness of the primal and dual solution sets $\xsolset$, $\ysolset$.
	
	Next, we utilize the result in \cite[Theorem 4.5.1]{drusvyatskiy2017many}, which is a restatement of the result in  \cite[Section 4]{sturm2000error}. Let us focus on the primal inequality \eqref{eq: primalg}. The optimal solution set of \eqref{p} is 
	$\xsolset = {\mathcal{L}} \cap \mathbb{S}_{+}^n$ where ${\mathcal{L}} :\,=\{X\mid \inprod{C}{X}=\pval,\, \Amap{X}=b\}$. Since the sublevel set $S_2$ is compact, the result in \cite[Theorem 4.5.1]{drusvyatskiy2017many} ensures that for some $c_1>0$, and $d_1>0$, there holds the inequality $\dist(X,\xsolset)^{2^{d_1}}\leq c_1 \dist(X,\mathcal{L})$ for any $X\succeq 0$. Here the number $d_1$ is called the singularity and is bounded by $n$ \cite[Lemma 3.6]{sturm2000error}.  Since $\dist(X,\mathcal{L})\leq c_2\left(\abs{\inprod{C}{X}-\inprod{C}{\xsol}}+\twonorm{\Amap{X}-b}\right)$ for some $c_2>0$ as $\mathcal{L}$ is an affine space, we have shown the inequality \eqref{eq: primalg}. In addition, if strict complementarity holds, then $d_1 \leq 1$ due to \cite[Section 5]{sturm2000error}. A similar argument applies to the dual inequality \eqref{eq: dualg} using \cite[Theorem 4.5.1]{drusvyatskiy2017many}, compactness of $\ysolset$, and that  $\dist(Z,\mathbb{S}^n_+)\leq n\max\{\lambda_{\max}(-Z),0\}$ for any $Z\in \mathbb{S}^n$. 
\end{proof}

Given $F(y_t)-F(\ysol)$ is converging to have zero objective gap, the above growth bound ensures $\dist(y_t,\ysolset)$ converges to zero. However, the corresponding rate of convergence would depend on the generic exponent $\zeta_1$. The following three lemmas provide direct relationships (without dependence on $\zeta_1$ or $\zeta_2$) between the spectral bundle method's convergence on the penalized dual formulation and the primal-dual feasibility and optimality of its iterates $X_t$ and $y_t$. Utilizing these bounds, our analysis of $(\rpast,\rcurrent)$-SpecBM can then focus on showing convergence in the penalized dual objective gap. For ease of notation, we utilize the shorthand $\bar{F}_t:\, = \bar{F}_{(V_t,\bar{X}_t)}$ to denote the spectral bundle method's approximation of $F$ at iteration $t$. 
\begin{lemma}[Primal Feasibility]\label{lem:primal-feas}
	At every descent step $t$, $(\rpast,\rcurrent)$-SpecBM has
	$$ X_{t}\succeq 0, \quad\text{and}\quad \|b - \Amap X_{t}\|^2\leq \frac{2\rho}{\beta}(F(y_{t}) - F(\ysol )). $$
\end{lemma}
\begin{proof}
According to the definition of $X_t$ in \eqref{eq: X_tdefinition},	we have $X_{t} =\eta^\star_{t}\bar{X}_t+V_t S^\star _{t}V_t^\top$.
	Since $\eta\geq 0$ and $S^\star_t\succeq 0$ by construction in \eqref{eq: subproblemAMSpectralbundleMethod}, $X_t$ is 
	positive semidefinite. 
	
	The first-order optimality condition for minimizing \eqref{eq: subproblemAMSpectralbundleMethod} ensures
	\begin{equation}\label{eq: linearFeasibilityEq1}
	-b + \Amap X_t  = \rho (y_{t}-y_{t+1}).
	\end{equation}
	Hence $\twonorm{-b + \Amap X_t }^2  = \rho^2\twonorm{ (y_{t}-y_{t+1})}^2$. The difference $y_{t}-y_{t+1}$ can be bounded as follows by the penalized dual objective value gap, completing the proof,
	$$ \frac{\rho}{2}\|y_{t+1} - y_t\|^2 \leq F(y_{t})-\bar F_{t}(y_{t+1}) \leq \frac{F(y_{t})-F(y_{t+1})}{\beta} \leq \frac{F(y_{t})-F(\ysol)}{\beta}$$
	where the first inequality follows as $y_{t+1} = z_{t+1}$ minimizes
	$\bar F_{t}(z)+\frac{\rho}{2}\|z-y_t\|^2$
	and the second follows from the definition of a descent step.
\end{proof}
\begin{lemma}[Dual Feasibility]\label{lem:dual-feas}
	At every descent step $t$, provided $\alpha \geq 2 D_{\xsolset}$, $(\rpast,\rcurrent)$-SpecBM has
	$$ \lambda_{\min}(C - \Ajmap y_{t+1})\geq \frac{-(F(y_{t}) - F(\ysol))}{D_{\xsolset}}.$$
\end{lemma}
\begin{proof}
	Strong duality ensures that for any $\xsol\in\xsolset$, one has $\inprod{C}{\xsol}=\inprod{b}{\ysol}$, or equivalently $\inprod{\xsol}{Z(\ysol)}=0$. Hence
	\begin{equation*} 
	\begin{aligned}
	\langle b, y_{t+1}-\ysol \rangle & = \langle \Amap\xsol, y_{t+1}-\ysol \rangle\\
	& = \langle \xsol, \Ajmap(y_{t+1}-\ysol)\rangle \\
	& = \langle \xsol, Z(\ysol)-Z(y_{t+1})\rangle\\
	& \leq -\nucnorm{\xsol}\min\{\lambda_{\min}(C - \Ajmap y_{t+1}),0\}.
	\end{aligned}
	\end{equation*}
	Since $\ysol$ minimizes \eqref{eq: penaltySDP}, 
	we have
	\begin{equation*} 
	\begin{aligned}
	F(y_{t}) - F(\ysol) \geq F(y_{t+1}) - F(\ysol) & = \langle -b, y_{t+1}-\ysol \rangle -\alpha\min\{\lambda_{\min}(C - \Ajmap y_{t+1}),0\}\\
	& \geq -\|\xsol\|\min\{\lambda_{\min}(C - \Ajmap y_{t+1}),0\}
	\end{aligned}
	\end{equation*}
	where the last inequality uses that $\alpha \geq 2D_{\xsolset} \geq 2\nucnorm{\xsol}$. 
Since $\xsol$ is arbitrary, we have the claimed feasibility bound.
\end{proof}
\begin{lemma}[Primal-Dual Optimality]\label{lem:optimality}
	At every descent step $t$, provided $\alpha \geq 2 D_{\xsolset}$, $(\rpast,\rcurrent)$-SpecBM has
	$$\langle b, y_{t+1} \rangle - \langle C, X_t\rangle \leq \frac{\alpha}{D_{\xsolset}}(F(y_{t}) - F(\ysol)) + \sqrt{\frac{2\rho}{\beta}(F(y_{t}) - F(\ysol))}\ D_{y_0}$$
	and below by
	$$
	\langle b, y_{t+1} \rangle - \langle C, X_t\rangle \geq -\frac{1-\beta}{\beta}(F(y_{t}) - F(\ysol)) - \sqrt{\frac{2\rho}{\beta}(F(y_{t}) - F(\ysol))}\ D_{y_0}.$$
\end{lemma}
\begin{proof}
	The standard duality analysis shows the primal-dual objective gap equals
	\begin{equation*} 
	\begin{aligned}
	\langle b, y_{t+1} \rangle - \langle C, X_t\rangle &= \langle \Amap X_t, y_{t+1} \rangle - \langle C, X_t\rangle + \langle b-\Amap X_t, y_{t+1}\rangle\\
	&= \langle X_t, \Ajmap y_{t+1} - C\rangle + \langle b-\Amap X_t, y_{t+1}\rangle.
	\end{aligned}
	\end{equation*}
	Notice that the second term here is bounded above and below as
	$$ |\langle b-\Amap X_t, y_{t+1}\rangle| \leq \sqrt{\frac{2\rho}{\beta}(F(y_{t}) - F(\ysol))}\ \|y_{t+1}\|\leq \sqrt{\frac{2\rho}{\beta}(F(y_{t}) - F(\ysol))}\ D_{y_0}$$
	using Lemma~\ref{lem:primal-feas} and that $\|y_{t+1}\|\leq D_{y_0}$ as $F(y_{t+1})\leq F(y_t)\leq F(y_0)$.
	Hence we only need to bound the first term above, $\langle X_t, \Ajmap y_{t+1} - C\rangle$, showing that the spectral bundle method approaches satisfying complementary slackness.
	
	An upper bound on this inner product follows from Lemma~\ref{lem:dual-feas} as
	$$ 
	\langle X_t, \Ajmap y_{t+1} - C \rangle \leq -\nucnorm{X_t}\lambda_{\min}(C - \Ajmap y_{t+1}) \leq \frac{\nucnorm{X_t}(F(y_{t})-F(\ysol))}{D_{\xsolset}}.
	$$
	Combining the above with $\tr(X_t) \leq \alpha$ by construction, we have  
	$$
	\langle b, y_{t+1} \rangle - \langle C, X_t\rangle \leq \frac{\alpha}{D_{\xsolset}}(F(y_{t}) - F(\ysol)) + \sqrt{\frac{2\rho}{\beta}(F(y_{t}) - F(\ysol))}\ D_{y_0}.
	$$
	A lower bound on this inner product follows as
	\begin{equation*} 
	\begin{aligned}
	\frac{1-\beta}{\beta}(F(y_{t}) - F(y_{t+1})) 
	&\geq F(y_{t+1}) - \bar F_{t}(y_{t+1})\\
	& = -\alpha\min\{\lambda_{\min}(C-\Ajmap y_{t+1}),0\} + \langle C, X_{t}\rangle -\langle \Amap X_{t}, y_{t+1}\rangle\\
	& \geq \langle X_{t}, C - \Ajmap y_{t+1}\rangle,
	\end{aligned}
	\end{equation*}
	where the first inequality follows from the definition of a descent step, the equality follows 
	from the definition of $\bar{F}_t$ and the optimality of $X_t$ in 
	\eqref{eq: subproblemAMSpectralbundleMethod}. Hence
	\begin{equation*} 
	\begin{aligned}
	\langle b, y_{t+1} \rangle - \langle C, X_t\rangle &\geq -\frac{1-\beta}{\beta}(F(y_{t}) - F(\ysol)) - \sqrt{\frac{2\rho}{\beta}(F(y_{t}) - F(\ysol))}\ D_{y_0}.
	\end{aligned}
	\end{equation*}
\end{proof}

\subsection{Proof of Theorem \ref{thm: sublinearates}}\label{sec: SublinearRates}
At some iteration $t$ of $(\rpast,\rcurrent)$-SpecBM, let $v_+$ be a top eigenvector of $\lambda_{\max}(\Amap^* z_{t+1}-C)$ if $\lambda_{\max}(\Amap^* z_{t+1}-C)>0$ 
and be zero otherwise. Then denote $g_{t+1}=-b + \alpha \Amap (v_+v_+^\top) \in \partial F(z_{t+1})$ as the subgradient corresponding to this maximum eigenvector and $s_{t+1}= -\rho (z_{t+1}-y_t)\in\partial \bar F_t(z_{t+1})$ as the aggregate subgradient, certifying optimality of~\eqref{eq: subproblemAMSpectralbundleMethod}.
For the existing proximal bundle method convergence rates to apply (see Section~\ref{subsec:bundle-rates}), we need to verify conditions~\eqref{eq:cond1}, \eqref{eq:cond2}, and \eqref{eq:cond3} hold with $\bar{f}_{t}= \bar F_t$.  
\newcontent{Given these conditions, Theorem~\ref{thm:prox-bm-convergence} ensures $(\rpast,\rcurrent)$-SpecBM has penalized dual objective gap $F(y_t)-F(\ysol)$ converging at a rate of $\bigO(1/\epsilon^3)$, or $\bigO(1/\epsilon)$ whenever quadratic growth holds (e.g., whenever strict complementarity holds by Lemma~\ref{lem: qg}).} Then our claimed results on primal feasibility, dual feasibility, and primal-dual optimality follow by applying Lemmas \ref{lem:primal-feas}, \ref{lem:dual-feas}, and \ref{lem:optimality}. 

\newcontent{
\subsubsection{Verifying~\eqref{eq:cond1}, \eqref{eq:cond2}, and \eqref{eq:cond3}}} Recall the spectral bundle method's model approximates $\{X \mid \langle X, I \rangle \leq \alpha, X \succeq 0\}$ at iteration $t+1$ by the spectral set
\[
\mathcal{W}_{t+1}:=\{\eta \bar{X}_{t+1} + V_{t+1}SV_{t+1}^\top \mid \eta \geq 0,\;S\in \symMat_+^{\bar r},\;\text{and}\;\eta\tr(\bar X_{t+1})  + \tr(S)\leq \alpha \},
\]
giving $\bar F_{t+1} (y) = \langle-b, y\rangle + \max_{X\in \mathcal{W}_{t+1}}\langle X, \Amap^*y-C\rangle$.

First we note that~\eqref{eq:cond1} is immediate for $(\rpast,\rcurrent)$-SpecBM since its model always lower bounds the true objective~\eqref{eqn: dualpenaltyobjectiverewrite} as $\mathcal{W}_{t+1}\subseteq \{X \mid \langle X, I \rangle \leq \alpha, X \succeq 0\}$.

Next we verify~\eqref{eq:cond2}. Since $V_{t+1}$ spans $v_+$, some vector $s$ has $V_{t+1}s = v_+$. Consequently considering $\eta=0$ and $S=\alpha ss^\top$ shows $\alpha v_+v_+^\top\in\mathcal{W}_{t+1}$ and so $$\bar F_{t+1}(y) \geq \langle-b, y\rangle + \langle \alpha v_+v_+^\top, \Amap^*y-C\rangle = F(z_{t+1}) + \langle g_{t+1}, y-z_{t+1}\rangle$$ holds with $g_{t+1}=-b + \alpha \Amap (v_+v_+^\top)$.

Finally, we verify~\eqref{eq:cond3}.
By the optimality condition of \eqref{eq: subproblemAMSpectralbundleMethod}, and definition of $X_t,z_{t+1}$, we know that 
\begin{align}
-b + \Amap X_t  & = \rho (y_t-z_{t+1})=s_{t+1} \label{eq: optimalityConditionEqsubproblemAMSpectralbundleMethod}\\ 
\bar{F}_{t}(z_{t+1}) & = \inprod{-b}{z_{t+1}} +\inprod{X_t}{\Amap^*z_{t+1}-C}.\label{eq: F_{t+1}(z_{t+1})} 
\end{align}
Similar to the reasoning for~\eqref{eq:cond2}, we first show $X_t$ lies in $\mathcal{W}_{t+1}$: To see this,
recall that $V_{t+1}$ was selected to span the $r_p$ top eigenvector directions of $S_t^\star$ given by $V_tQ_1$.
Then there is an $S$ such that $V_{t+1}SV_{t+1}^\top = V_tQ_1\Lambda_1 Q_1^\top V_t^\top$.\footnote{Indeed, one can take $S =V_{t+1}^\top V_tQ_1\Lambda_1 Q_1^\top V_t^\top V_{t+1}$ as $V_{t+1}$ spans the columns of $V_tQ_1$.} This choice of $S$ alongside $\eta=1$ 
has $X_t = \eta \bar{X}_{t+1} + V_{t+1}SV_{t+1}^\top $ due to the updating scheme \eqref{eqn: updatebarX_t+1FC} of $\bar{X}_{t+1}$ 
and definition of $X_t$ in \eqref{eq: X_tdefinition}. This choice of $S$ is feasible because $S\succeq 0$ as $\Lambda_1\succeq 0$, and 
\begin{equation}
\begin{aligned}  \label{eq: X*intheNewSettrace}
\eta \tr (\bar{X}_{t+1})+\tr(S) & = \tr(\eta^\star _t\bar{X}_t)+ \tr(V_tQ_2\Lambda_2Q_2^\top V_t^\top) + \tr(V_tQ_1\Lambda_1 Q_1^\top V_t^\top)\\
&=\eta_t^\star\tr(\bar{X}_t) +\tr(V_tS^\star_tV_t^\top) \leq \alpha
\end{aligned} 
\end{equation} 
where the first equality above relies on the definition of $\bar{X}_{t+1}$ and that $\tr(S)= \tr(V_{t+1}SV_{t+1}^\top)$ 
(because $V_{t+1}$ has orthonormal columns), and the last inequality is due to $V_t$ having orthonormal columns and $\eta^\star _t$ and $S^\star _t$ satisfying the constraint $\eta^\star \tr(\bar{X}_t) +\tr (S^\star _t)\leq \alpha$ by construction. Thus  $$\bar F_{t+1}(y) \geq \langle-b, y\rangle + \langle X_t, \Amap^*y-C\rangle = \bar F_t(z_{t+1}) + \langle s_{t+1}, y-z_{t+1}\rangle$$ holds with $s_{t+1}= -\rho (z_{t+1}-y_t)$.

\subsection{Proof of Theorem~\ref{thm: linear convergence of Block SBM under the extra condition strict complementarity}} \label{sec: Proof of Local linear convergence of Block bundle method}
\newcontent{In this section, we first discuss the needed bounds on $T_0$ and $\rho$ for our linear convergence analysis to apply. In the following subsections, we prove the following central pair of lemmas which directly imply Theorem~\ref{thm: linear convergence of Block SBM under the extra condition strict complementarity}. Namely, under  appropriate selections of $T_0$ and $\rho$, the model $\bar{F}_{t}$ becomes quadratically close to the true penalized dual objective (see Lemma~\ref{lem: quadraticallyAccurate}). Consequently, every iteration is a descent step, linearly contracting towards optimality (see Lemma~\ref{lem: linearlyConvergent}).
\begin{lemma}\label{lem: quadraticallyAccurate}
Under the assumptions and notations in Theorem \ref{thm: linear convergence of Block SBM under the extra condition strict complementarity}, there is some $\eta>0$ (independent of $\epsilon$) such that for $t\geq T_0$, we have
\begin{equation}\label{eq: quadraticAccurateModel}
\begin{aligned}
\bar{F}_{t}(z)\leq F(z)\leq \bar{F}_{t}(z) + \frac{\eta}{2}\twonorm{z-y_t}^2 \qquad \text{ for all } z\in\mathbb{R}^\ncons.
\end{aligned}
\end{equation}
\end{lemma}

\begin{lemma}\label{lem: linearlyConvergent}
Suppose \eqref{eq: quadraticAccurateModel} holds at iterate iteration $t$, then for any $\rho\geq \eta$, Algorithm \ref{alg:bundle} with any choice of $\beta \in (0,\frac{1}{2}]$ will take a descent step satisfying
\begin{equation}\label{eq: approximationModelStep4}
\disttwonorm (y_{t+1},\ysolset)  \leq \sqrt{\frac{\rho}{2\gamma_1 +\rho}} \disttwonorm (y_t,\ysolset).
\end{equation}
\end{lemma}
}
We note that if we assume $y_t = z_t$ always, then Lemma \ref{lem: linearlyConvergent} can be derived using a combination of the proofs for prox-linear method in 
\cite{drusvyatskiy2018error,drusvyatskiy2019efficiency}. The reader might find the detailed procedure in Appendix \ref{sec: relationshipDimaAdrian}.
Our proof here is self-contained, directly employs the quadratic growth of $F$ and the quadratic closeness of the model $\bar{F}$, and shows that the descent step is taken, i.e., $y_t = z_t$. 
\subsubsection{Discussion on the bounds on $T_0$ and $\rho$} Denote the gap parameter as $\delta := \inf_{\ysol \in \ysolset}\max_{r\leq r_d} \lambda_{r}(-Z(\ysol ))-\lambda_{r+1}(-Z(\ysol))$ and
the quadratic growth 
parameter for $F$ from Lemma \ref{lem: qg} as $\gamma_1>0$.
The gap parameter $\delta$ is nonzero 
from the definition of $r_d$, the compactness of $\ysolset$, and continuity of the 
function $\max_{r\leq r_d} \lambda_{r}(Z(\cdot ))-\lambda_{r+1}(Z(\cdot))$. When the dual solution is unique, we have
$\delta =  \lambda_{r_d}(-Z(\ysol))-\lambda_{r_d+1}(-Z(\ysol))$. 
With these notations, the constant $\eta$ in Lemma \ref{lem: quadraticallyAccurate} is $\eta = 4\alpha \opnorm{\Amap^*}^2\max\left\{\frac{72\sup_{\ysol \in \ysolset}\opnorm{2Z(\ysol)}}{\delta^2},\frac{9(8\sqrt{2}+16)}{\delta}\right\}$ \newcontent{(see the proof of Lemma \ref{lem: quadraticallyAccurate} in Section \ref{sec: Proof of the quadratic accurate model} for details)}.

Let the number $T_0$ be the first descent step such that for all $t\geq T_0$, $Z(y_t)$ is $\delta/3$ close to the solution set 
$Z(\ysolset)=\{Z(\ysol)\mid \ysol \in \ysolset\}$. Using the $\bigO(1/\epsilon)$ convergence rate from Theorem \ref{thm: sublinearates} and quadratic growth from Lemma \ref{lem: qg}, this must hold for all
\begin{equation} \label{eq: T0estimate}
t\geq T_0= \bigO\left(\frac{\opnorm{\Amap^*}^2}{\delta^2\gamma_1}\right).
\end{equation} 
Indeed, for any $y\in \real^m$, by picking a solution $\ysol\in \ysolset$ closest to $y$, we have 
\begin{equation}
\begin{aligned} 
    \fronorm{Z(y)-Z(\ysol)} = \fronorm{\Amap^*(y-\ysol)}
    \leq& \opnorm{\Amap^*}\twonorm{y-\ysol} 
       = \opnorm{\Amap^*}\dist(y,\ysolset)\\
    \leq& 
    \opnorm{\Amap^*} \sqrt{\left(F(y)-F(\ysol)\right)/\gamma_1}.
\end{aligned} 
\end{equation}
From the above inequality, we see that the condition  $F(y_t) -F(\ysol)\leq \frac{\delta^2\gamma_1}{9\opnorm{\Amap^*}}$ ensures that \newcontent{$\lambda_{r_d}(Z(y_t))\geq \frac{\delta}{3}$}. Such condition is satisfied for any $y_t$ with $t\geq T_0$ by the $\bigO(1/\epsilon)$ convergence rate from Theorem \ref{thm: sublinearates} and our choice of $T_0$ in \eqref{eq: T0estimate}. 

We require the regularization parameter $\rho$ be chosen larger than $\eta$,  i.e.,
\begin{equation}\label{eq: rhoEstimate}
\rho \geq 4\alpha \opnorm{\Amap^*}^2\max\left\{\frac{72\sup_{\ysol \in \ysolset}\opnorm{2Z(\ysol)}}{\delta^2},\frac{9(8\sqrt{2}+16)}{\delta}\right\}.
\end{equation}

\subsubsection{Proof of Lemma~\ref{lem: quadraticallyAccurate}} \label{sec: Proof of the quadratic accurate model}
Without loss of generality, we have $z_t = y_t$ (that is, the previous step was a descent step). Define the $r$-th spectral plus set of a matrix $X\in\symMat^{\dm}$ with $\lambda_{r}(X)-\lambda_{r+1}(X)>0$ as  $\faceplus r(X):=\{VSV^{\top}\mid\tr(S)\leq1,S\succeq0,S\in\symMat^{r}\}$ where $V\in \RR^{\dm \times r}$ is the matrix formed by the orthonormal eigenvectors of $X$ corresponding to its 
$r$ largest eigenvalues. The following lemma, proved in Section~\ref{sec: proofOfImportantLemma}, shows these top eigenvectors give a quadratically accurate model.
\begin{lemma}\label{lem: importantLemmaQuadraticAccurateModel}
	Suppose $X\in\symMat^{\dm}$ has eigenvalues $\lambda_{r}(X)-\lambda_{r+1}(X)=\delta$ and denote the $\Lambda_{r,\dm}=\max\{|\lambda_{r+1}(X)|,|\lambda_\dm (X)| \}$. Then for any $Y\in \symMat^{\dm}$, the quantity $f_X(Y):\,=\max\{\lambda_1(Y),0\}-\max_{W\in \faceplus r (X)}\inprod{W}{Y}$ satisfies that 
	\begin{equation}
	\begin{aligned} \label{eq: quadraticAccurateModelLemma}
	0\leq f_X(Y) \leq \frac{8\fronorm{Y-X}^{2}\Lambda_{r,\dm}}{\delta^{2}}+\frac{(8\sqrt{2}+16)\fronorm{Y-X}^{2}}{\delta}.
	\end{aligned} 
	\end{equation}
\end{lemma}
This lemma shows that the function $\max_{W\in \faceplus r (X)}\inprod{W}{Y}$ has captured the nonsmooth part of the $\max\{\lambda_1(Y),0\}$ and is accurate to $\max\{\lambda_1(Y),0\}$ up to second order. This result is key to establishing~\eqref{eq: quadraticAccurateModel} for all $t\geq T_0$ in the following two sections (first assuming a unique dual solution for ease and then in general).

\myparagraph{Unique solution case} First suppose the dual solution $\ysol$ is unique and the corresponding dual slack is denoted as $\zsol$. In this case, our choice of $T_0$ ensures $y_t$ is sufficiently close to $\ysol$ such that $\opnorm{Z(y_t)-\zsol}\leq \frac{\delta}{3}$ where $\delta$ is the $r_d$-th eigengap of $-\zsol$. 
Then from Weyl's inequality, we know the $r_d-$th eigengap of $-Z(y_t)$, $\lambda_{r_d}(-Z(y_t))-\lambda_{r_d+1}(-Z(y_t))$, is at least $\frac{\delta}{3}$, and $\opnorm{Z(y_t)}\leq 
2\opnorm{\zsol}$.

Let $V\in \RR^{\dm \times r_d}$ denote the matrix formed by the eigenvectors corresponding to the $r_d$ largest eigenvalue of $-Z(y_t)$. We find that 
\begin{equation*}
\begin{aligned} 
F(y)-\bar{F}_t(y)& =\alpha \max\{\lambda_{\max}(-Z(y)),0\} -  \max_{\eta \alpha +\tr(S)\leq \alpha , \eta \geq 0, S\in \symMat_+^{\rcurrent}}\inprod{\eta\bar{X}+V_tSV_t^\top}{-Z(y)} \\ 
&\leq \alpha \max\{\lambda_{\max}(-Z(y)),0\} -  \max_{\tr(S)\leq \alpha, S\in \symMat_+^{r_d}}\inprod{VSV^\top}{-Z(y)} \\
& = \alpha \left( \max\{\lambda_{\max}(-Z(y)),0\} - \max_{W\in \faceplus {r_d} (-Z(y_t))}\inprod{W}{-Z(y)} \right) \\
& \leq  \alpha\left( \frac{72\fronorm{Z(y_t)-Z(y)}^{2}\opnorm{2\zsol}}{\delta^{2}}+\frac{9(8\sqrt{2}+16)\fronorm{Z(y_t)-Z(y)}^{2}}{\delta} \right)\\
&\leq 2\alpha \opnorm{\Amap^*}^2\max\{\frac{72\opnorm{2\zsol}}{\delta^2},\frac{9(8\sqrt{2}+16)}{\delta}\}\twonorm{y-y_t}^2,
\end{aligned} 
\end{equation*}
where the first inequality restricts the spectral set considered since $\rcurrent \geq r_d$ by assumption and the second inequality applies Lemma \ref{lem: importantLemmaQuadraticAccurateModel}.
Combining the fact that $\bar{F}_t$ lower bounds $F(y)$ by construction, we see the model $\bar{F}_t$ is indeed quadratically accurate with $\eta =  4\alpha \opnorm{\Amap^*}^2\max\{\frac{72\opnorm{2\zsol}}{\delta^2},\frac{9(8\sqrt{2}+16)}{\delta}\}$ in \eqref{eq: quadraticAccurateModel}.

\myparagraph{Multiple dual solutions case} Now we generalize the above reasoning to when $\mathcal{Y}_\star$ contains multiple points. Recall we defined $\delta$ as
\begin{equation}
\begin{aligned} \label{eq: definitionOfdelta}
\delta = \inf_{\ysol \in \ysolset}\max_{r\leq r_d} \lambda_{r}(-Z(\ysol))-\lambda_{r+1}(-Z(\ysol))
\end{aligned} 
\end{equation}
which is nonzero from the definition of $r_d$, the compactness of $\ysolset$, and continuity of the 
function $\max_{r\leq r_d} \lambda_{r}(Z(\cdot))-\lambda_{r+1}(Z(\cdot))$. Hence if $\dist(Z(y_t),Z(\ysolset))$ is 
less than a third of $\delta$, then there is an $r$ and $\ysol\in \ysolset$, such that $-Z(\ysol)$ is no more than 
$\frac{\delta}{3}$ away from $Z(y_t)$, and has $\lambda_{r}(-Z(\ysol))-\lambda_{r+1}(-Z(\ysol))\geq \delta$. 
Hence, we can repeat previous argument for the case of unique dual solution and replace $r_d$ and $\opnorm{\zsol}$ by $r$ and $2\sup_{\ysol \in \ysolset}\opnorm{Z(\ysol)}$ respectively. Thus the model $\bar{F}_t$ is quadratically accurate in \eqref{eq: quadraticAccurateModel} with 
$
\eta =4\alpha \opnorm{\Amap^*}^2\max\{\frac{72\sup_{\ysol \in \ysolset}\opnorm{2Z(\ysol)}}{\delta^2},
$
$
\frac{9(8\sqrt{2}+16)}{\delta}\}
$ 
as stated in \eqref{eq: rhoEstimate}.

\subsubsection{Proof of Lemma~\ref{lem: linearlyConvergent}}\label{sec: Linear convergence under a quadratic accurate model}
Suppose~\eqref{eq: quadraticAccurateModel} is satisfied for some $\eta>0$ at iteration $t$. Without loss of generality, $\eta=\rho$ since we require $\eta\leq \rho$. We first show $(\rpast,\rcurrent)$-SpecBM must take a descent step for $\beta\leq \frac{1}{2}$. We know the minimizer 
$z_{t}^\star$ of $\bar{F}_{t}(z) + \frac{\rho}{2}\twonorm{z-y_t}^2$ satisfies that for any $z\in \RR^\ncons$
\begin{equation}\label{eq: threePointInequality}
\begin{aligned} 
\bar{F}_t(z_t^\star) +\frac{\rho}{2}\twonorm{z_t^\star-y_t}^2 + \frac{\rho}{2}\twonorm{z^\star_t-z}^2\leq \bar{F}_t(z) + \frac{\rho}{2}\twonorm{z- y_t}^2,
\end{aligned} 
\end{equation} 
since $\bar{F}_t(z)+\frac{\rho}{2}\twonorm{z-y_t}^2$ is $\rho$-strongly convex. Setting $z=y_t$ and \eqref{eq: quadraticAccurateModel} shows
\begin{equation}\label{eq: approximationModelStep1}
F(y_t)-\bar{F}_t(z_t^\star)\geq \rho \twonorm{z_t^\star-y_t}^2\geq 0,\quad \text{and}\quad \frac{\rho}{2}\twonorm{z_t^\star-y_t}^2 \leq F(y_t)-F(z_{t}^\star ).
\end{equation}
Using sequentially that $\beta\leq 1/2$, the quadratic bound~\eqref{eq: quadraticAccurateModel} and then~\eqref{eq: approximationModelStep1} shows
\begin{equation}\label{eq: descentStepASBMLocalLinearConvergence}
\begin{aligned}
\beta \left(F(y_t)-\bar{F}_t(z_t^\star)\right)\leq  \frac{1}{2}\left(F(y_t)-\bar{F}_t(z_t^\star)\right) 
&\leq\frac{1}{2}\left(F(y_t)-F(z_t^\star)\right)+\frac{\rho}{4}\twonorm{y_t-z_t^\star}^2\\
&\leq  F(y_t)-F(z_{t}^\star).
\end{aligned} 
\end{equation}
Hence, we see the method will indeed take a descent step and $y_{t+1}=z_t^\star$ is in the sublevel set defined by $\{y\mid F(y)\leq F(y_0)\}$. 

Now we show this descent step contracts the distance to $\ysolset$, yielding linear convergence. Considering $z=\ysol$ for any $\ysol\in \ysolset$ in \eqref{eq: threePointInequality} and using 
\eqref{eq: quadraticAccurateModel} ensures 
\begin{equation}\label{eq: approximationModelStep2}
\begin{aligned} 
F(z^\star_t)\leq F(\ysol) +\frac{\rho}{2}\left(\twonorm{\ysol -y_t}^2-\twonorm{z_t^\star-\ysol}^2\right). 
\end{aligned} 
\end{equation} 
Now recall the quadratic growth of $F$  (derived from Lemma \ref{lem: qg}) that there is a $\gamma_1>0$ such that for all $z\in \{y\mid F(y)\leq F(y_0)\}$, 
\[
F(z)-F(\ysol)\geq \gamma_1 \disttwonorm^2 (z,\ysolset).
\]
Hence combining this with \eqref{eq: approximationModelStep2}, we find that 
\begin{equation}\label{eq: approximationModelStep3}
\begin{aligned} 
\gamma_1 \disttwonorm^2 (z_t^\star,\ysolset) &\leq \frac{\rho}{2}\left(\twonorm{\ysol -y_t}^2-\twonorm{z_t^\star-\ysol}^2\right) \\ 
\implies  \left(\gamma_1 +\frac{\rho}{2}\right)\disttwonorm^2 (z_t^\star,\ysolset)  & \leq \frac{\rho}{2}\twonorm{\ysol -y_t}^2 \\ 
\implies \disttwonorm^2 (z^\star_t, \ysolset) & \leq \frac{\rho}{2\gamma_1 +\rho} \twonorm{\ysol-y_t}^2. 
\end{aligned} 
\end{equation}
Setting $\ysol$ to be the point in $\ysolset$ nearest to $y_t$ shows $y_{t+1}=z_t^\star$ satisfies the recurrence
\begin{equation*}
\disttwonorm (y_{t+1},\ysolset)  \leq \sqrt{\frac{\rho}{2\gamma_1 +\rho}} \disttwonorm (y_t,\ysolset), 
\end{equation*}
ensuring convergence occurs geometrically, contracting by a factor of $\sqrt{\frac{\rho}{2\gamma_1 +\rho}} <1$.

\section{Proof of Lemma \ref{lem: importantLemmaQuadraticAccurateModel}}\label{sec: proofOfImportantLemma}
Lastly, we provide a proof of Lemma \ref{lem: importantLemmaQuadraticAccurateModel}. In addition
to the Frobenius norm bound, we also provide an operator two norm bound \eqref{eq:qudraticaccuracyopnorm}. 

Recall the assumption that $\lambda_{r}(X)-\lambda_{r+1}(X)=\delta>0$ for
some $\delta>0$. Let $V\in\real^{n\times r}$ be an orthonormal matrix
formed by the $r$ eigenvectors corresponding to the top $r$ eigenvalues.
Recall $r$-th spectral plus set $\faceplus r(X):=\{VSV^{\top}\mid\tr(S)\leq1,S\succeq0,S\in\symMat^{r}\}.$
Note that for any orthonormal $O\in \real^{r\times r}$, replacing $V$ by $VO$ produces the same spectral set $\faceplus r(X)$. 

For any $Y\in\symMat^{\dm},$ since $\max\{\lambda_{1}(Y),0\}=\max_{W\succeq0,\tr(W)\leq1,}\inprod WY,$
we see the following inequality always holds as $\{W|W\succeq0,\tr(W)\leq1\}\supset\faceplus r(X)$:
\begin{equation}
\max\{\lambda_{1}(Y),0\}  \geq\max_{W\in\faceplus r(X)}\inprod WY.
\end{equation}
Define the error $f_X(Y)$ as 
\begin{equation}
f_{X}(Y)  =\lambda_{1}(X)-\max_{W\in\faceplus r(X)}\inprod WY.
\end{equation}
We always have $f_{X}(Y)\geq0$ as previously argued. If $\lambda_{1}(Y)<0$,
then $\max\{\lambda_{1}(Y),0\}=0$ and hence $Y \preccurlyeq 0$. Thus
the approximation $\max_{W\in\faceplus r(X)}\inprod WY=0$ as well.
Hence we may only consider the case $\lambda_{1}(Y)>0$ in the following.
Let $v$ be the eigenvector with two norm $\twonorm v=1$ corresponding
to the largest eigenvalue $\lambda_{1}(Y)$, then 
\begin{align*}
f_{X}(Y) & =\lambda_{1}(Y)-\max_{W\in\faceplus r(X)}\inprod WY  =\min_{W\in\faceplus r(X)}\inprod{vv^{\top}-W}Y\\
& =\min_{W\in\faceplus r(X)}\underbrace{\inprod{vv^{\top}-W}{Y-X}}_{R_{1}}+\underbrace{\inprod{vv^{\top}-W}X}_{R_{2}}.
\end{align*}
To analyze $R_{1}$ and $R_{2}$, we define some notation first. Denote
$V'\in\mathbb{\mathbb{R}}^{\dm\times r}$ to be the orthonormal matrix
formed by the eigenvectors corresponding to the top $r$ eigenvalues
of $Y$. Moreoever, let $v$ denote the first column of $V'$. Also denote
$F\in\real^{n\times(n-r)}$ to be an orthonormal matrix formed by the
rest eigenvectors of $X.$ So the eigenvalue decomposition of $X$
is $X=V\Lambda_{1}V^{\top}+F\Lambda_{2}F^{\top},$ for some diagonal
$\Lambda_{1}\in\symMat^{r}$ and $\Lambda_{2}\in\symMat^{(\dm-r)\times(n-r)}$. Let the 
matrix $O^{\star}\in\real^{r\times r}:O^{\star}\in\arg\min_{OO^\top =I}\fronorm{VO-V'}$.
Below, we set $V_O = VO^{\star}$.

We bound the $R_{2}$ term first. We may choose $W=VV^{\top}vv^{\top}VV^{\top}$
here. With such a choice, $R_{2}$ equals the following:
\begin{align*}
R_{2} & =\inprod{vv^{\top}-W}X  =\inprod{vv^{\top}}X-\inprod WX\\
& \overset{(a)}{=}\inprod{vv^{\top}}{V\Lambda_{1}V^{\top}+F\Lambda_{2}F^{\top}}-\inprod{VV^{\top}vv^{\top}VV^{\top}}{V\Lambda_{1}V^{\top}+F\Lambda_{2}F^{\top}}\\
& \overset{(b)}{=}\inprod{vv^{\top}}{V\Lambda_{1}V^{\top}}+\inprod{vv^{\top}}{F\Lambda_{2}F^{\top}}-\inprod{vv^{\top}}{V\Lambda_{1}V^{\top}}\\
& \overset{(c)}{=}\inprod{V'e_{1}(V'e_{1})^{\top}}{F\Lambda_{2}F^{\top}}.
\end{align*}
Here we use the eigenvalue decomposition of $X$ in step $(a)$. Step $(b)$ uses the fact that $V$ has orthonormal columns and that
$V^{\top}F=0$ as the columns are orthonormal. Step $(c)$ uses the
fact that $v$ is the first column of $V'$.

Let
the error between $V_O$ and $V'$ be given by $E=V'-V_O$ and let $e_{1}\in\real^{r}$ be the vector with first entry $1$ and all other entries $0$. 
Using these, we upper bound $R_{2}=\inprod{V'e_{1}(e_{1}V')^{\top}}{F\Lambda_{2}F^{\top}}$ as
\begin{equation} \label{eq:T2maxlambda}
\begin{aligned}
R_{2} 
& =\inprod{(V_O+E)e_{1}e_{1}^{\top}(V_O+E)^{\top}}{F\Lambda_{2}F^{\top}}\\
& \overset{(a)}{=}\inprod{Ee_{1}e_{1}^{\top}E^{\top}}{F\Lambda_{2}F^{\top}} 
\\
&\overset{(b)}{\leq}\nucnorm{Ee_{1}e_{1}^{\top}E^{\top}}\opnorm{F\Lambda_{2}F^{\top}}
\\
& \overset{(c)}{=}\opnorm{Ee_{1}e_{1}^{\top}E^{\top}}\opnorm{F\Lambda_{2}F^{\top}}
\\
& \overset{(d)}{\leq}\opnorm{Ee_{1}}^{2}\opnorm{\Lambda_{2}}
\\
& \overset{(e)}{\leq}\opnorm E^{2}\opnorm{\Lambda}.
\end{aligned}
\end{equation} 
Here we use the fact $V_O^{\top}F=0$ in step $(a)$. Step $(b)$ is
due to the H\"{o}lder's inequality. Step $(c)$ uses the fact that for
rank $1$ matrix, the Frobenius norm is the same as its operator norm.
Step $(d)$ uses the submultiplicity of operator two norm. The last
step $(e)$ uses the fact operator norm of $e_{1}$ is $1$. 

Next we bound $R_{1}$. Considering $W=VV^{\top}vv^{\top}VV^{\top}=V_OV_O^\top vv^\top V_OV_O^\top$,
the difference $vv^{\top}-W$ is 
\begin{align*}
vv^{\top}-W & =V'e_{1}(V'e_{1})^{\top}-V_OV_O^{\top}vv^{\top}V_OV_O^{\top}\\
& =V'e_{1}(V'e_{1})^{\top}-V_OV_O^{\top}V'e_{1}(V'e_{1})^{\top}V_OV_O^{\top}\\
& =(V_O+E)e_{1}e_{1}^{\top}(V_O+E)^{\top}-V_OV_O^{\top}(V_O+E)e_{1}e_{1}^{\top}(V_O+E)^{\top}V_OV_O^{\top}\\
& =Ee_{1}e_{1}^{\top}V_O^{\top}+V_Oe_{1}e_{1}^{\top}E^{\top}+Ee_{1}e_{1}E^{\top}\\
& -V_OV_O^{\top}Ee_{1}e_{1}^{\top}V_O^{\top}-V_Oe_{1}e_{1}^{\top}E^{\top}V_OV_O^{\top}-V_OV_O^{\top}Ee_{1}e_{1}^{\top}EV_OV_O^{\top}.
\end{align*}
Hence, using the fact the nuclear norm of rank one matrix is the same
as operator norm, the nuclear norm of $vv^{\top}-W$ is bounded by
\begin{align*}
\nucnorm{vv^{\top}-W} 
 \leq& \opnorm{Ee_{1}e_{1}^{\top}V^{\top}_O}+\opnorm{V_Oe_{1}e_{1}^{\top}E^{\top}}
 +\opnorm{Ee_{1}e_{1}E^{\top}}\\
 &+\opnorm{V_OV_O^{\top}Ee_{1}e_{1}^{\top}V_O^{\top}}
 +\opnorm{V_Oe_{1}e_{1}^{\top}E^{\top}V_OV_O^{\top}}\\
 &+\opnorm{V_OV_O^{\top}Ee_{1}e_{1}^{\top}EV_OV_O^{\top}}\\
 \overset{(a)}{\leq}&4\opnorm E+2\opnorm E^{2}.
\end{align*}
Here in step $(a)$, we use the fact that $\opnorm{e_{1}e_{1}^{\top}}\leq1$
and $\opnorm V\leq1.$ Using H\"{o}lder's inequality again, the first term $R_{1}$ is bounded by
\begin{equation}\label{eq: T1maxlambda}
R_{1} 
=\inprod{vv^{\top}-W}{Y-X} 
\leq\nucnorm{vv^{\top}-W}\opnorm{Y-X}
\leq\left(4\opnorm E+2\opnorm E^{2}\right)\opnorm{Y-X}.
\end{equation}

Now combining (\ref{eq: T1maxlambda}) and (\ref{eq:T2maxlambda}),
we find that $f_{X}(Y)$ is upper bounded by 
\begin{align*}
f_{X}(Y) & \leq\opnorm{\Lambda_{2}}\opnorm E^{2}+\left(4\opnorm E+2\opnorm E^{2}\right)\opnorm{Y-X}.
\end{align*}

Let us consider two cases:
\begin{enumerate}
	\item First consider the Frobenius norm. The Frobenius bound \cite[Theorem 2]{yu2015useful} applied to $E$ ensures
	that 
	\begin{align*}
	\fronorm E & \leq\frac{2\sqrt{2}\fronorm{Y-X}}{\delta}.
	\end{align*}
	Hence in this case, we have for all $Y\in\symMat^{\dm}$
	\begin{align*}
	f_{X}(Y) & \leq\frac{8\fronorm{Y-X}^{2}\opnorm{\Lambda_{2}}}{\delta^{2}}+\frac{8\sqrt{2}\fronorm{Y-X}^{2}}{\delta}+\frac{16\fronorm{Y-X}^{3}}{\delta^{2}}.
	\end{align*}
	\item Second consider the operator norm. Using \cite[Theorem 2]{yu2015useful} again, we have the operator norm of $E$ bounded
	by 
	\begin{align*}
	\opnorm E & \leq\frac{2\sqrt{2}\sqrt{r}\opnorm{Y-X}}{\delta}.
	\end{align*}
	In this case, the function $f_{X}(Y)$ is upper bounded by 
	\begin{align}
	f_{X}(Y) & \leq\frac{8r\opnorm{Y-X}^{2}\opnorm{\Lambda_{2}}}{\delta^{2}}+\frac{8\sqrt{2}\sqrt{r}\opnorm{Y-X}^{2}}{\delta}+\frac{16r\opnorm{Y-X}^{3}}{\delta^{2}}.\label{eq:fxYopeartornormboundglobalwithr}
	\end{align}
	If $\opnorm{Y-X}\leq\delta$, then using it for the term $\frac{16r\opnorm{Y-X}^{3}}{\delta^{2}}$,
	we have 
	\begin{align}
	f_{X}(Y) & \leq\frac{8r\opnorm{Y-X}^{2}\opnorm{\Lambda_{2}}}{\delta^{2}}+\frac{(8\sqrt{r}+16r)\opnorm{Y-X}^{2}}{\delta}.\label{eq:fxYoperatornormboundwithr}
	\end{align}
\end{enumerate}
Still, we have not reached a globally quadratically accurate model. Let us show that the function $f_{X}(Y)$ is always bounded by a linear
difference. Note that we have
$\max_{W\in\faceplus r(X)}\inprod WY = \max\{\lambda(V^\top YV),0\}$.
We decompose $f_{X}(Y)$ into two terms: 
\begin{align*}
f_{X}(Y) & =\max\{\lambda_{1}(Y),0\}-\max\{\lambda_{1}(V^{\top}YV),0\}\\
& =\underbrace{\max\{\lambda_{1}(Y),0\}-\max\{\lambda_{1}(X),0\}}_{R_{1}}+\underbrace{\max\{\lambda_{1}(X),0\}-\max\{\lambda_{1}(V^{\top}YV),0\}}_{R_{2}}.
\end{align*}
For the term $R_{1}$, we note the function $\max\{x,0\}$ for any
$x\in\real$ is $1$-Lipschitz with respect to the norm
$|x|.$ Thus the term $R_{1}$ is bounded by 
\begin{equation*}
|R_{1}| 
\leq|\lambda_{1}(Y)-\lambda_{1}(X)|
\leq\opnorm{Y-X}.
\end{equation*}

For the second term $R_{2}$, we note that $\lambda_{1}(X)=\lambda_{1}(V^{\top}XV$)
because of the definition of $V$. Hence, using the same reasoning,
we have 
\begin{equation*}
|R_{2}| 
\leq\opnorm{V^{\top}XV-V^{\top}YV}
\leq\opnorm{X-Y}
\end{equation*}
where the last line is due to submultiplicity of operator two norm.
Hence, we see the error function $f_{X}(Y)$ is always bounded by
\begin{align*}
|f_{X}(Y)| & \leq2\opnorm{X-Y}.
\end{align*}

The inequality (\ref{eq:fxYoperatornormboundwithr}) tells us that
when $\opnorm{X-Y}\leq\delta$, we have $f_{X}(Y)\leq\frac{8r\opnorm{Y-X}^{2}\opnorm{\Lambda_{2}}}{\delta^{2}}+\frac{(8\sqrt{2r}+16r)\opnorm{Y-X}^{2}}{\delta}.$
Now if $\opnorm{X-Y}\geq\delta$, then it follows that $\frac{2\opnorm{X-Y}^{2}}{\delta}\geq2\opnorm{X-Y}$.
Hence, the model $\max_{W\in\faceplus r(X)}\inprod WY$ is always
quadratically accurate: for all $Y\in\symMat^{\dm},$
\begin{align}
\text{0\ensuremath{\leq f_{X}\left(Y\right)\leq}} & \min\left\{ \frac{8r\opnorm{Y-X}^{2}\opnorm{\Lambda_{2}}}{\delta^{2}}+\frac{(8\sqrt{2r}+16r)\opnorm{Y-X}^{2}}{\delta},\frac{2\opnorm{X-Y}^{2}}{\delta}\right\} .\label{eq:qudraticaccuracyopnorm}
\end{align}
The same argument applies to the Frobenius norm case, and we reach 
\begin{align}\label{eq: qudraticaccuracyFrobeniusNorm}
f_{X}(Y) & \leq\frac{8\fronorm{Y-X}^{2}\opnorm{\Lambda_{2}}}{\delta^{2}}+\frac{(8\sqrt{2}+16)\fronorm{Y-X}^{2}}{\delta}.
\end{align}
	\section{Numerics}\label{sec: numerics}







In this section, we first present numerical experiments demonstrating (i) sublinear convergence generically for the spectral bundle method under a range of configurations and (ii) once the conditions listed in Theorem~\ref{thm: linear convergence of Block SBM under the extra condition strict complementarity} hold, convergence speeds up (to linear convergence). Subsequently, we show substantial speedups in both time and space complexity are achievable utilizing a sketching technique, enabling the spectral method to effectively solve much larger problem instances.

\subsection{Max-Cut and Matrix Completion Experiments}\label{sec: maxcutmatrixcompenumerics}

We consider two common SDP problems, matrix completion and max-cut, whose formulations are stated in Table~\ref{tb: SpecBundleMethodperformance}. As discussed in Section \ref{sec: discussion On low-rankness}, both of these families of problems typically have low-rank primal optimal solutions.

We consider the following instances of these semidefinite programs: For max-cut, we take $L$ as the Laplacian of the graph G24 in \cite{Gset} with $2000$ vertices. 
For matrix completion, $\Omega$ denotes the set of indices of the observed entries of the underlying rank $3$ matrix
$\trux \in \RR^{1000\times 1000}$. Here $\trux = WW^\top$ where $W\in \RR^{1000\times 3}$ with each 
entry following the Rademacher distribution. Each entry of $\trux$ is observed with 
probability $p=0.04$. Both problems have decision variable size $2000\times 2000$. 

For both problems, we initialize with $X_0,y_0,z_0$ all zero. For max-cut, we set $\alpha =2\dm$, $\rho=0.5,\beta = 0.25$ and run for $200$ iterations, and for matrix completion, we set $\alpha = 4 \nucnorm{\trux},\rho =5,\beta= 0.25$ and run for $100$ iterations. The subproblem \eqref{eq: subproblemAMSpectralbundleMethod} is solved via Mosek \cite{mosek2010mosek}. Likewise, the optimal value $\pval$ and primal solution $\xsol$ for max-cut is obtained through Mosek \cite{mosek2010mosek}, whereas for matrix completion, we set 
$\pval = 2\nucnorm{\trux}$ and $\xsol = \begin{bmatrix}
\trux & \trux \\
\trux & \trux 
\end{bmatrix}$ for matrix completion. Such a choice of $\xsol$ indeed solves matrix completion SDP 
with high probability \cite{candes2009exact}. 
Let the rank of the optimal solution be $\rsol = \rank(\xsol)$, which is $18$ for max-cut and 
$3$ for matrix completion.

We consider two configurations of the parameters $\rcurrent$ and $\rpast$: (i) $\rcurrent =1$ while $\rpast = \rsol-2$, $\rsol-1$, and $\rsol$, and (ii) $\rpast =0$, $\rcurrent = \rsol-1$, $\rsol$, and $\rsol+1$. In the first setting, we primarily accumulate past information with $\rpast$ on the order of the rank of the primal optimal solution, while computing only one new eigenvector per iteration. In the second setting, the method retains no additional past information (beyond the aggregation $\bar X_t$), relying primarily on the current $r_c$ eigenvectors.

\myparagraph{Experiment Results} Table \ref{tb: SpecBundleMethodperformance} shows the accuracy of the last iterates in terms of primal and dual optimality and feasibility. The dual optimality (dual opt.), primal optimality (primal opt.), and primal feasibility (primal feas.) 
are defined as $\frac{F(y)-\dval}{\abs{\dval}}$, $\abs{\frac{\inprod{C}{X}-\pval}{\pval}}$, and $\frac{\norm{\Amap X-b}}{\norm{b}}$ respectively. We find that primal feasibility tends to be worse than dual optimality by one or two orders of magnitude, while primal optimality is usually of the same order. Slower convergence in primal feasibility aligns with expectations based on our lemmas in Section~\ref{sec: analtical conditon}, as primal feasibility $\|\Amap X-b\|$ is only guaranteed to be on the order of the square root of dual optimality.

\begin{table}
\begin{center}
	{\small\begin{tabular}{|cllll|}
		\hline 
		Problem & $(\rpast,\rcurrent)$  & Dual Opt. & Primal Opt. & Primal Feas.\\ 
		\hline 
		\textbf{Matrix Completion}                                 	 & $(1,1)$  & $0.08836$   & 0.01950 &  0.02383 \\ 
		\multirow{5}{*}{$
			\displaystyle
			\begin{aligned}
			&{\text{max}}
			& & -\inprod{I}{W_1}-\inprod{I}{W_2} \\
			& \text{s.t.}
			& &  X_{ij} = \trux_{ij},\, (i,j)\in \Omega \\
			&  & &
			\begin{bmatrix}
			W_1 & X \\ X^\top & W_2
			\end{bmatrix}\succeq 0.
			\end{aligned}
			$}                                                        &   $(0,2)$ & 0.07774 & 0.01077 &  $0.01020$ \\ 
		& $(2,1)$ & 0.04868 & $0.008485$ & 0.02086  \\ 
		&  $(0,3)$ & $8.014$e-7   & $3.843$e-8 & $8.220$e-5 \\ 
		& $(3,1)$ &  $6.853$e-6 & $2.729$e-5  & $5.582$e-4  \\ 
		& $(0,4)$  & $2.800$e-6 & $1.880$e-6 &  $1.437$e-4\\ 
		\hline \hline 
		\textbf{Max-cut}                                 	& $(16,1)$ & $0.02078$ & $0.01530$ & 0.2324 \\ 
		\multirow{5}{*}{$
			\displaystyle
			\begin{aligned}
			&\mbox{max} & & \inprod{L}{X}\\
			&\mbox{s.t.} & & \diag(X) =  \ones \\
			& & & X\succeq 0
			\end{aligned}
			$
		}
		&  $(0,17)$ &  $7.697$e-6 &  $4.546$e-6 & $2.239$e-4 \\ 
		& $(17,1)$ & $0.01261$ & $0.02124$ & 0.1907 \\ 
		&  $(0,18)$ & $1.789$e-7 & $6.862$e-7& $1.049$e-4  \\ 
		& $(18,1)$ & $0.01716$ & $0.01640$ & 0.2258\\ 
		& $(0,19)$ &  $6.776$e-8 & $8.144$e-9  & $8.963$e-5\\ 
		\hline 
	\end{tabular} }
	\caption{The final accuracy of $y_t$ and $X_t$ reached by the spectral bundle method for matrix completion and max-cut problem under different configurations of $(\rpast,\rcurrent)$.}\label{tb: SpecBundleMethodperformance}
\end{center}
\end{table} 

Figure~\ref{fig: penalizedObjectiveValue} shows the evolution of dual objective value $F$ in \eqref{eq: penaltySDP} as each method runs. The spectral bundle methods performance across these configurations and problem settings tends to agree with our theories predictions. As shown in Figure \ref{fig:figure Matrix Completion} and Figure \ref{fig:figure Max-cut}, in general, when $\rcurrent <\rsol$, the spectral bundle method converges sublinearly (with the exception $(3,1)$-SpecBM for matrix completion). Once $\rcurrent\geq \rsol$, the method converges quickly for both max-cut and matrix completion
as expected from our Theorem~\ref{thm: linear convergence of Block SBM under the extra condition strict complementarity}. 
As shown in Table \ref{tb: SpecBundleMethodperformance},  whenever $\rcurrent \geq \rsol$, the method solves both problems in terms of dual optimality to moderately high accuracy ($\sim 10^{-8}$). Most instances with $\rcurrent <\rsol$ only achieved a moderate accuracy ($\sim 10^{-2}$). We suspect the limitation to $10^{-8}$ accuracy is due to the inaccuracy in the eigenvalue computations or the subproblem solver for \eqref{eq: subproblemAMSpectralbundleMethod}.

\begin{figure}[t]
	\begin{subfigure}[(a)]{.48\textwidth}
		\centering
		\includegraphics[width=1\linewidth, height= 0.25\textheight]{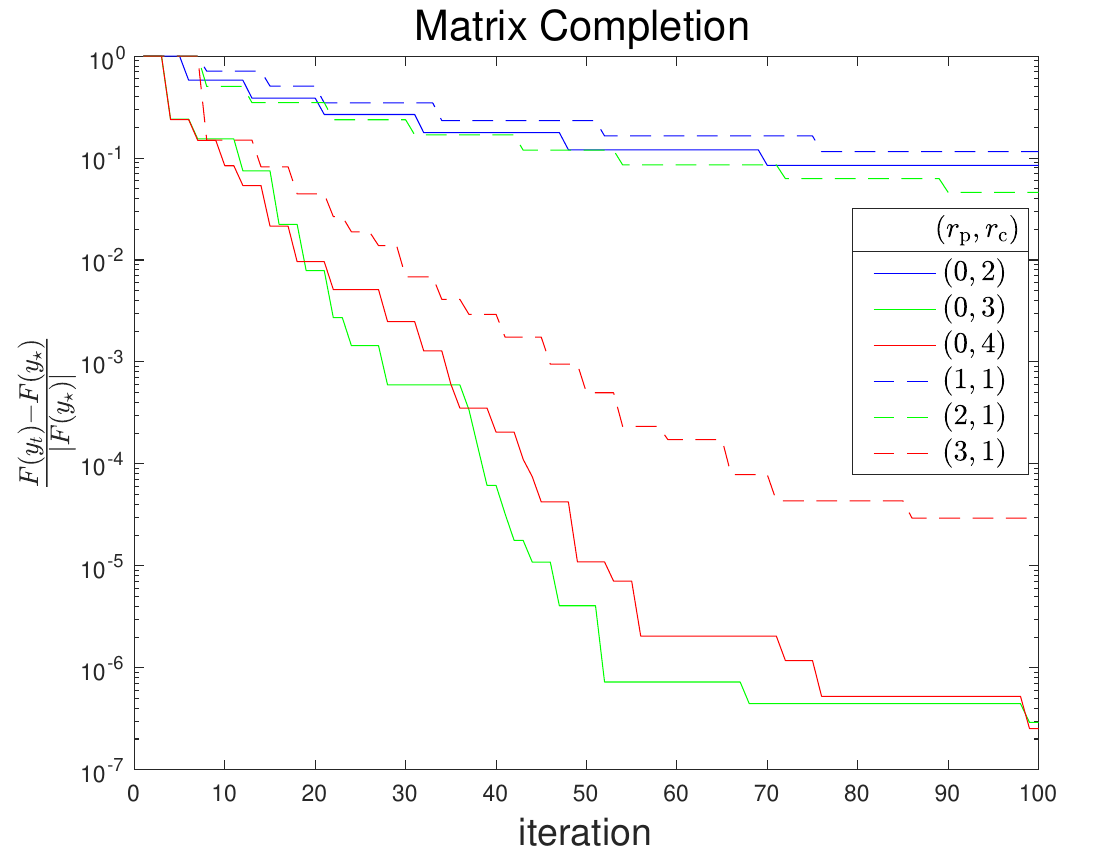}
		\caption{Matrix Completion}
		\label{fig:figure Matrix Completion}
	\end{subfigure}%
	\hspace{5pt}
	\begin{subfigure}[(b)]{.48\textwidth}
		\centering
		\includegraphics[width=1\linewidth, height= 0.25\textheight]{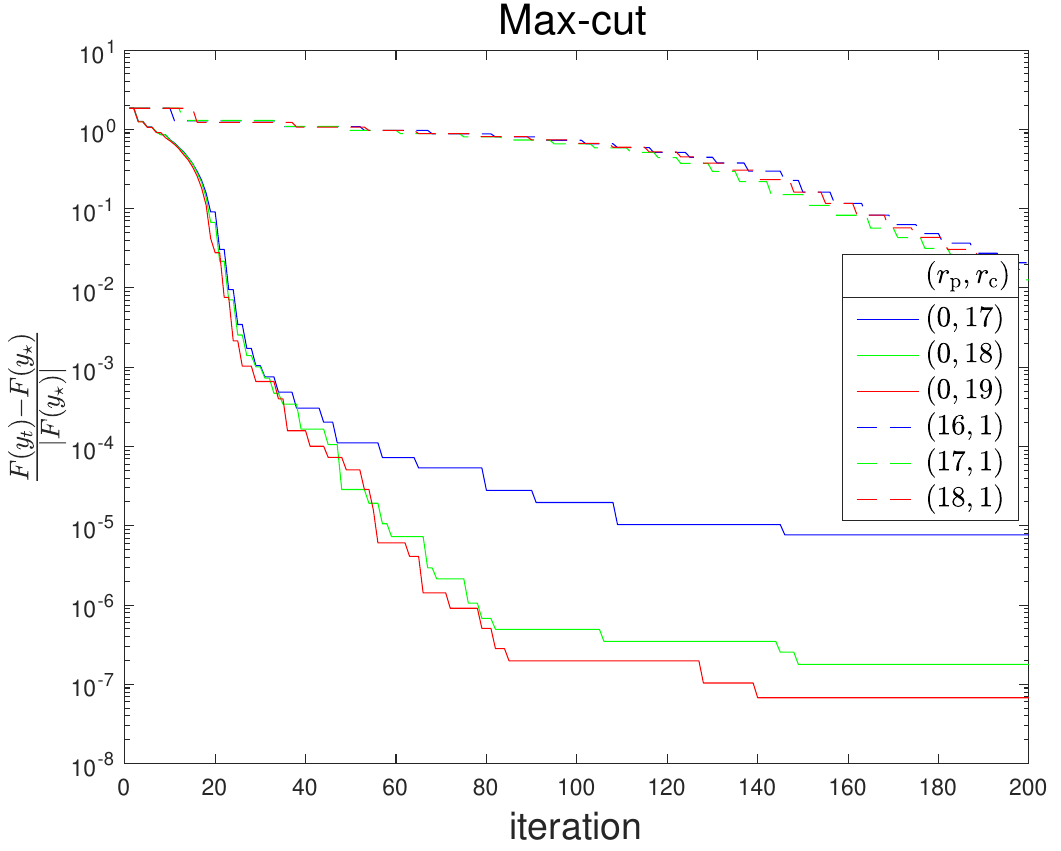}
		\caption{Max-cut}
		\label{fig:figure Max-cut}
	\end{subfigure}
	\caption{The evolution of the relative penalized dual objective value $\frac{F(y_t)-F(\ysol)}{|{F(\ysol)}|}$ for different configurations of the $(\rpast,\rcurrent)$ on problems of size $2000\times 2000$.}
	\label{fig: penalizedObjectiveValue}
\end{figure}

\subsection{Matrix Sketching for the Spectral Bundle Method} \label{subsec:sketching}
A particular bottleneck in solving large scale SDPs is storing the primal matrix $X_t$. Here we show how to avoid storing this iterate by introducing the matrix sketching idea developed in \cite{tropp2017fixed,tropp2017practical,yurtsever2017sketchy} and demonstrate its usage on the previous max-cut and matrix completion instances. Applying $(\rpast,\rcurrent)$-SpecBM with such a sketching procedure, we are able to solve a matrix completion SDPs with several billion decision variables (up to $(1.6\times 10^5)\times(1.6\times 10^5)$) in only a few minutes.

\newcontent{
We detail the matrix sketching procedure in Algorithm \ref{alg:sketch}.
%
%
\begin{algorithm2e}[t]
  \SetAlgoNoLine
	\KwData{Iteration number $T$, dimension $n$, an integer $r>0$, the sequence of $V_tQ_2\Lambda_2 Q_2^\top V_t^\top,t=0,\dots,T$, and $V_{T+1},S_{T+1}^\star$ in Algorithm \ref{alg:spec-bundle}\;}
  {Sample $\Psi \in \real^{n\times R}$ with an $R\geq 3r+1$ and $\Psi_{ij}\overset{\text{iid}}{\sim} N(0,1)$\;} 
  {Initialize $\bar{Y}_0=0\in \real^{n\times R}$\;}
 \For{$t=0,\dots,T-1$}{
Update $\bar{Y}_{t+1} =  V_tQ_2\Lambda_2Q_2^\top (V_t^\top\Psi)+ \eta_t^\star  \bar{Y}_t $
 }
 {Compute $Y_{T} = V_{T}S^\star _{T}(V_{T}^\top\Psi)+ \eta_{T}^\star  \bar{Y}_{T}$\;}
 {Reconstruct $\hat{X}_T:= Y_T(\Psi^\top Y_T)^{\dagger}Y_T^\top$\;}
	\caption{Matrix sketching procedure}
	\label{alg:sketch}
\end{algorithm2e}
The method requires an integer $r>0$, which represents either an estimate of the true rank of the primal solution or the user's computational/storage budget for managing larger matrices. 

The algorithm first draws a random matrix $\Psi \in \real^ {n \times R}$ with i.i.d.~normal entries.
Denote the number of total iteration as $T$. The main idea of the method is that using $\Psi$, we can form a low rank sketch of $\bar X_t$, denoted by $\bar Y_t$, as
\begin{equation}\label{eqn: onetimeSketch}
\bar Y_t = \bar X_t\Psi \in \real^{\dm \times R}.
\end{equation}
Using the update formula $\bar{X}_{t+1} = \eta^\star_t \bar{X}_t +V_tQ_2\Lambda_2Q_2^\top V_t^\top$,
we can obtain $\bar Y_t$ as done in Algorithm \ref{alg:sketch} on Line 4. 
%

The retrieve the primal matrix $X_T$,
we form the sketch matrix $Y_t = X_t \Psi= \eta_t^\star \bar{Y}_t + V_{T}S^\star _{T}(V_{T}^\top\Psi)$ 
using the relationship that $X_{t} = \eta^\star_{t}\bar{X}_{t}+V_{t} S^\star_{t} V_{t}^\top$.
The matrix 
$X_t$ is then reconstructed using $Y_t$ via the last line of the algorithm where
the notation $(\Psi^\top Y_T)^{\dagger}$ is the pseudo-inverse of $\Psi^\top Y_T$. Note that $\hat{X}_T$ is positive semidefinite since $X_T$ is and $\Psi^\top Y_T = \Psi^\top X_T\Psi_T$. A numerical stable implementation of the last line can be found in \cite[Algorithm 3]{tropp2017fixed}, which outputs $(U_T,\Lambda_T)$ such that $\hat{X}_T=U_T \Lambda_T U_T^\top$ with $U_T \in \real^{n\times R}$ having orthonormal columns and a nonnegative diagonal $\Lambda_t$. 
Note that one can then store $\hat{X}_T$ via the factors $(U_T,\Lambda_T)$ rather than forming $\hat{X}_T$ explicitly.
From \cite[Theorem 4.1]{tropp2017fixed}, we have the following guarantee:
\begin{equation}\label{eq: sketchbound}
    \mathbb{E}_{\Psi}\|X_T - \hat{X}_T\|_*\leq \frac{4}{3}\|X_T -[X_T]_r\|_*,
\end{equation}
where $\|\cdot\|_*$ is the nuclear norm and $[\cdot]_r$ is the best rank $r$ approximation in terms of Frobenius norm. Hence, if $X_T$ is close to a low rank matrix $\xsol$, then so long as $r\geq \rsol$, the approximation error $X_T-\hat{X}_T$ is small. 


Thus $(\rpast,\rcurrent)$-SpecBM combined with the matrix sketching procedure can avoid \emph{forming new iterates} $\bar{X}_t$ and $X_t$, which each occupies $\bigO(n^2)$ storage.
As discussed in the end of Section \ref{sec: ImportantsubproblemSolver}, we know that we can solve 
the subproblem \eqref{eq: subproblemAMSpectralbundleMethod} by storing $d_t = \Amap \bar{X}_t$, $c_t = \inprod{C}{\bar{X}_t}$, and $h_t=\tr(\bar{X}_t)$ rather than computing them directly from $\bar{X}_t$. Hence,
$(\rpast,\rcurrent)$-SpecBM combined with the matrix sketching idea described above can report a nearly optimal, low rank $X_T$ while  only using storage of size }
\begin{equation}\label{eq: storagesizeMsketching}
    \bigO(\underbrace{nr}_{\text{storing $\Psi$, $\bar{Y}_t$, and $Y_t$}} +\underbrace{m}_{\text{storing $d_t, \,c_t$, and $h_t$}}).
\end{equation}
The quantity $\bigO(nr + m)$ can be significantly smaller than $\bigO(n^2)$ for 
applications of SDP \eqref{p} when the rank estimate $r$ is small (constant or logarithmic with respect to $n$) and $m$ is on the order of $n$, see \cite{ding2019optimal} and \cite{yurtsever2017sketchy} for further discussion of storage benefits.

%
%

\subsubsection{Max-Cut and Matrix Completion Experiments with Sketching}
Continuing the previous experimental setup for max-cut and matrix completion, we demonstrate the usage of the matrix sketching procedure here.
First, in Figure \ref{fig: Xt rank}, we measure the numerical rank of $X_t$ (measured by the number of singular values larger than $10^{-2}$). 
We see that the intermediate rank of $X_t$ can be much larger than the primal optimal solution rank, $\rank(\xsol)$, even though we expect it will eventually converge to have rank equal to $\rank(\xsol)$.
This is a particularly relevant observation in justifying the use of sketching procedures as only a low-rank sketch of the primal solution matrix needs to be stored at any time. Alternative approaches, such as storing $X_t$ via a factorization (e.g., its eigenvalue decomposition), may still incur high storage costs due to the high rank of intermediate iterates. 
\footnote{
\newcontent{Careful readers might notice that for the case of Max-cut and SpecBM with $r_c =1$, the iterate is always low rank. A further investigation (not shown here) on the trace of $V_tS_t^\star V_t^\top$ and $\eta_t^\star \bar{X}_t$ shows that when $r_c=1$, $\eta_t^\star \bar{X}_t$  is negligible and $V_tS_t^\star V_t^\top$ dominates. This might be due to the design as we keep the important past spectral information explicitly in $V_tQ_1$ as in Algorithm \ref{alg:spec-bundle}. This observation might suggest to use $V_tS_t^\star V_t^\top$ alone to approximate $X_t$ rather than using the matrix sketching. However, the iterate $X_t$, in this case, does not produce a good approximation of $X_\star$.}
}

\begin{figure}[h]
	\begin{subfigure}[(a)]{.49\textwidth}
		\centering
		\includegraphics[width=1\linewidth, height= 0.25\textheight]{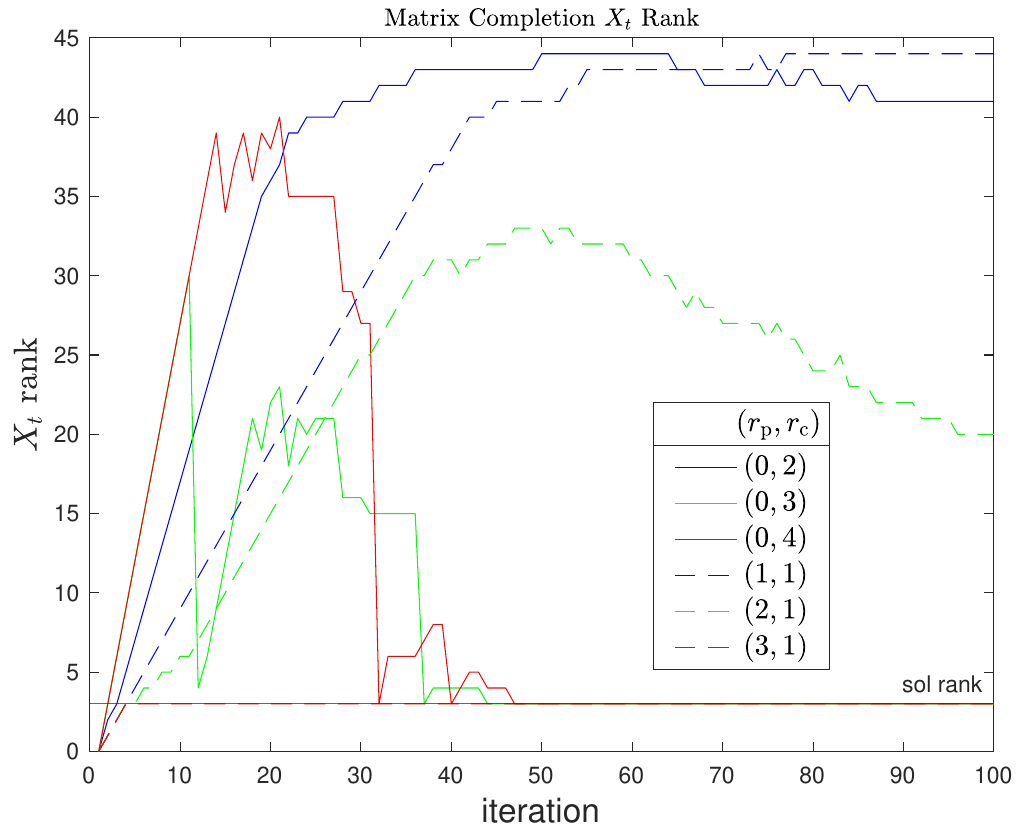}
		\caption{Matrix Completion where $\text{rank}(X_\star)=3$}
		\label{fig:figure Matrix Completion Xt rank}
	\end{subfigure}%
	\hspace{5pt}
	\begin{subfigure}[(b)]{.49\textwidth}
		\centering
		\includegraphics[width=1\linewidth, height= 0.25\textheight]{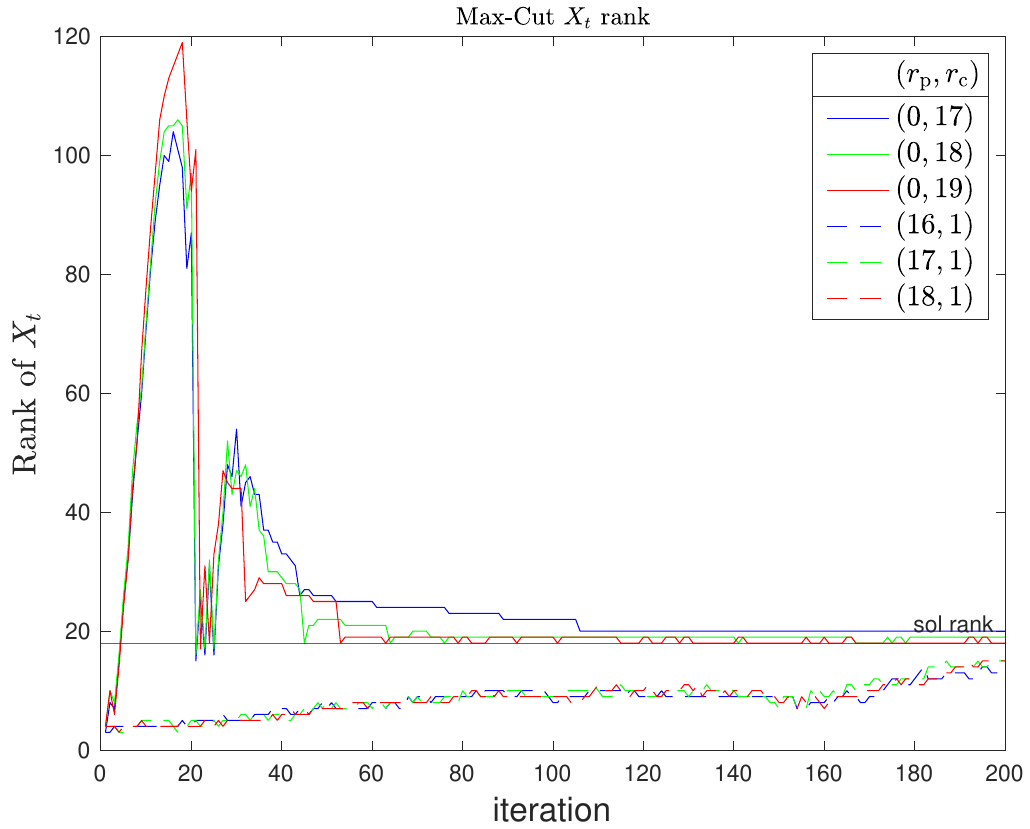}
		\caption{Max-cut where $\text{rank}(X_\star)=18$}
		\label{fig:figure Max-cut Xt rank}
	\end{subfigure}
	\caption{The evolution of the rank of the converging primal sequence $X_t$.}
	\label{fig: Xt rank}
\end{figure}
\begin{figure}[h]
	\begin{subfigure}[(a)]{.49\textwidth}
		\centering
		\includegraphics[width=1\linewidth, height= 0.25\textheight]{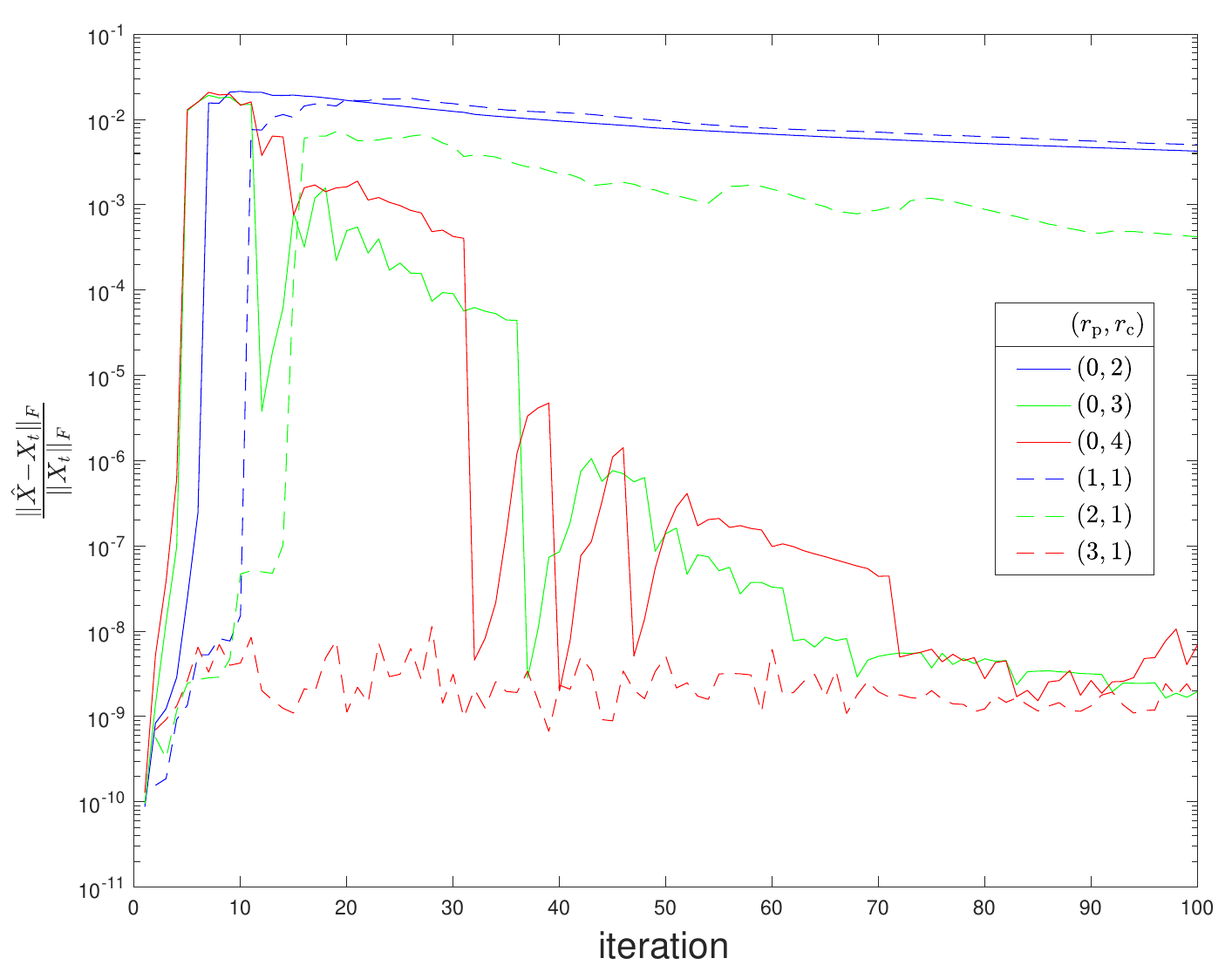}
		\caption{Matrix Completion where $\text{rank}(X_\star)=3$}
		\label{fig:figure Matrix Completion Xt diff}
	\end{subfigure}%
	\hspace{5pt}
	\begin{subfigure}[(b)]{.49\textwidth}
		\centering
		\includegraphics[width=1\linewidth, height= 0.25\textheight]{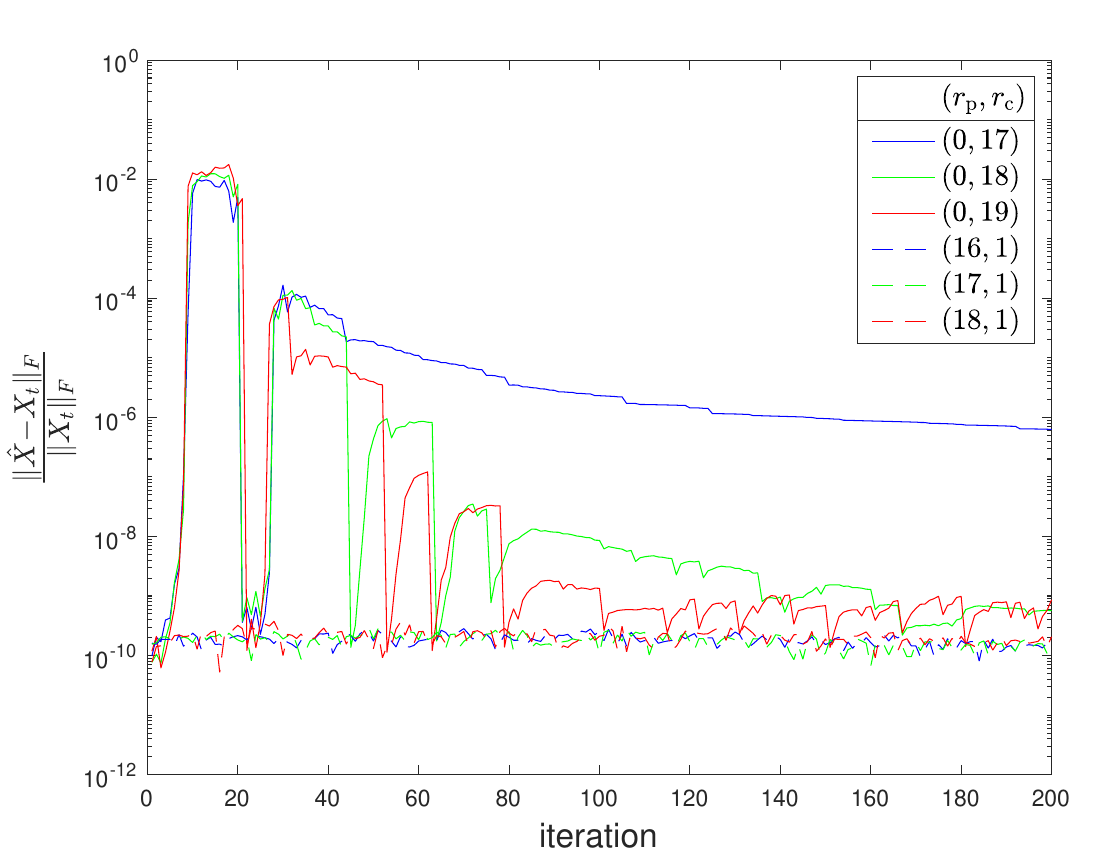}
		\caption{Max-cut where $\text{rank}(X_\star)=18$}
		\label{fig:figure Max-cut Xt diff}
	\end{subfigure}
	\caption{The evolution of the difference  $\frac{\fronorm{X_t-\hat{X}_t}}{\fronorm{X_t}}$ between the sketching primal sequence $\hat{X}_t$ and the primal sequence $X_t$.}
	\label{fig: Primal Sketching}
\end{figure}

Next, we measure the relative difference between the reconstruction $\hat{X}_t$ and the true iterate $X_t$. \newcontent{That is, we reconstruct $X_t$ in each iteration $t$ by setting the input $T=t$ in Algorithm \ref{alg:sketch}.} We set $R=10$ for the case of matrix completion and $R=60$ for the case of max-cut. Note that this is only for demonstration, in practice, one does not need to reconstruct $X_t$ in every iteration but only reconstruct it whenever needed. Figure \ref{fig: Primal Sketching} shows the potential for a large difference $\frac{\fronorm{X_t-\hat{X}_t}}{\fronorm{X_t}}$ at intermediate iterates, but only a negligible difference once the method is converging compared to the level primal optimality and feasibility of $X_t$.

\subsubsection{Large-scale Matrix Completion Experiments} Finally, we demonstrate that $(\rpast,\rcurrent)$-SpecBM coupled with the matrix sketching procedure described above (Sketching $(\rpast,\rcurrent)$-SpecBM) is able to solve much larger problem instances. To illustrate this, we compare it with Mosek \cite{mosek2010mosek}, SDPT3 \cite{toh2006implementation}, SketchyCGAL \cite{yurtsever2021scalable}, and
$(\rpast,\rcurrent)$-SpecBM. The first two are matured general purpose solvers. The third is designed for solving large scale SDPs.

Due to the ease of computing the optimal solutions with bounded rank for simulated matrix completion problems, we focus on it. We follow the setting of matrix completion in Section \ref{sec: maxcutmatrixcompenumerics} but vary the dimension of $X^\natural$ and the observation probability $p = \frac{100}{n}$.  The resulting primal matrix size $n\times n$ of the SDP ranges from $n=200$ to $n = 1.6\times 10^{5}$. For $(\rpast,\rcurrent)$-SpecBM and Sketching $(\rpast,\rcurrent)$-SpecBM, we set $\rpast = 0$ and $\rcurrent =4$ based on the previous strong performance when $n=2000$. We allow each method to run for $400$ seconds with up to $8$GB of memory. Table~\ref{tab: comparison} displays the relative recovery error $\frac{\fronorm{\hat{X}-X^\natural}}{{\fronorm{X^\natural}}}$ of each method, where $\hat{X}$ is extracted from the left top $\frac{n}{2}\times \frac{n}{2}$ block of the returned solution. 

From  Table~\ref{tab: comparison}, we see that the matured solvers (Mosek and SDPT3) are able to reach very high accuracy $10^{-10}\sim 10^{-11}$ but do not scale past $n=1600$. SketchyCGAL successfully scales to tackle problem of size $n=160000$. However, we note the recovery error degrades as $n$ grows. (For $n\geq 80000$, this error is worse than the trivial estimator $\hat{X}=0$). We see that SpecBM achieves a moderate accuracy $10^{-4}\sim 10^{-5}$ whenever the method did not exhaust its time and memory budget. However, the method does not scale up for problems of size $n\geq 20000$ due to the cost of managing the $n\times n$ matrix $X_t$. 

Sketching SpecBM achieves the best balance between accuracy and time. It is able to solve the SDP problem of size $n=160000$ to moderate accuracy $10^{-4}$. In terms of storage (which is not report here), we also observe the memory usage of the sketching method scales linearly with the dimension $n$ while SpecBM scales quadratically. 

\begin{table}[h]
\begin{center}
{\footnotesize
\begin{tabular}{||c | c c c  c c  ||} 
\hline 
$n$ & Mosek & SDPT3 & SketchyCGAL & SpecBM & Sketching SpecBM \\ 
\hline \hline 
200 & 1.0200e-10 & 3.9900e-10 & 3.1400e-04 & 9.4800e-05 & 4.1100e-05\\ 
400 & 8.2000e-10 & 7.2800e-09 & 5.9800e-04 & 8.4500e-05 & 2.6800e-05\\ 
800 & 8.9400e-10 & 9.5000e-11$^\star$ & 9.0400e-04 & 9.8600e-05 & 1.4800e-05\\ 
1600 & 5.3700e-11$^\star$ & $\infty$  & 0.0013 & 1.9000e-04 & 9.1500e-05\\ 
3200 & $\infty$ & $\infty$  & 0.0020 & 6.3600e-05 & 1.5900e-05\\ 
5000 & $\infty$  & $\infty$ & 0.0042 & 1.1000e-04 & 4.3400e-05\\ 
10000 & $\infty$ & $\infty$  & 0.0073 & 1.5300e-04 & 9.5600e-05\\ 
20000 & $\infty$  & $\infty$  & 0.1503 & $\infty$ & 1.0700e-05\\ 40000 & $\infty$  & $\infty$  & 0.1640 & $\infty$ & 1.3100e-04\\ 80000 & $\infty$ & $\infty$ & 1.3523 & $\infty$  & 1.2700e-04\\ 160000 & $\infty$  &$\infty$ & 1.5652 & $\infty$ & 1.5200e-04\\ 
\hline 
\end{tabular}
}
\end{center}
    \caption{Comparison of different solvers for the matrix completion problem 
    in Section \ref{sec: maxcutmatrixcompenumerics}. The symbol $\infty$ notes failure to finish within $400$ seconds or requiring over 8GB of memory. Both entries with $^\star$ used more than $400$ seconds but less than three hours.}
    \label{tab: comparison}
\end{table}

	\section{Discussion}\label{sec: discussion}
In this paper, we presented sublinear convergence rates for a family of spectral bundle methods and show the method speeds up to linear convergence with proper parameter choice and low-rank structural assumptions. We verify our theoretical results via numerical experiments and demonstrate their applicability to solving large-scale semidefinite programs. 

We conclude by presenting a few future directions, further building on the theoretical and practical effectiveness of spectral bundle methods:
\begin{itemize}
	\item \textbf{Handling more general constraints:} The problem format \eqref{p} only has equality and positive semidefiniteness constraints. Incorporating inequality 
	constraints and certain norm constraints such as $\norm{\Amap X-b}\leq \varepsilon$ for some $\varepsilon>0$ might be beneficial 
	for other semidefinite programming applications such as 
	stochastic block models with more than 2 blocks \cite{amini2018semidefinite} and noisy matrix completion \cite{candes2010matrix}. 
	It seems 
	straightforward to extend this work to these new settings by introducing additional dual variables or analyzing 
	new dual objectives.
	\item \textbf{Converging to the relative interior of the dual solution set $\ysolset$:} In Theorem \ref{thm: linear convergence of Block SBM under the extra condition strict complementarity},
	the rank estimate $\rcurrent$ needs to satisfy $\rcurrent\geq r_d$ instead of $\rcurrent\geq \rank(\xsol)$ assuming uniqueness of the primal solution. Though
	the quantity $r_d$ can be indeed larger than $\rank(\xsol)$ as shown in \cite[Theorem 5.1]{ding2020regularity}, $\rcurrent \geq \rank(\xsol)$ already ensures quick 
	convergence in our numerics. Based on the proof of 
	Theorem \ref{thm: linear convergence of Block SBM under the extra condition strict complementarity}, 
	this more general setting of linear convergence can be proved assuming the method converges to a dual solution that is in the relative interior of $\ysolset$. This is indeed what we observed by examining the dual slack matrices experimentally. Of course, this 
	cannot be guaranteed by the current algorithm design. Hence we pose the question of whether an algorithm can maintain our low per iteration complexity while always converging to the relative interior of the optimal solution set.
	\item \textbf{Adaptive choice of $\rho$ and $(\rpast,\rcurrent)$:} Our analysis assumes the choice of $\rho$ and $(\rpast,\rcurrent)$ is constant. Is it possible to analyze adaptively setting $\rho$ and $(\rpast,\rcurrent)$? An adaptive rule of $\rcurrent$ is of great practical importance as the prior information about the primal solution rank may not be available to the user. In \cite[eq. (40) and Remark 4]{oustry2000second}, two adaptive rules of $\rcurrent$ have been proposed. These rules might be combined with an upper bound on $\rcurrent$ to ensure the per iteration computation complexity does not explode. An adaptive choice of $\rpast$ may not be of critical importance given the existence of the aggregation, though we may simply use the adaptive rule of $\rcurrent$ for $\rpast$. Adaptive rules for updating $\rho$ have been considered in~\cite{DG2020}. 
 We leave theoretical and numerical investigations of these rules into future work.
	\item \newcontent{\textbf{Matrix sketching or not:} We require an external procedure matrix sketching to enhance the scalability of SpecBM. For the special case $(3,1)$-SpecBM for matrix completion, such an external procedure is not needed as $V_tS_t^\star V_t^\top$ approximates $X_\star$ well (not shown here). Further analysis and numerical investigation on this direction, especially combined with adaptive rank choice, is interesting and may reveal that SpecBM is self-sufficient for scalability.}
	\item \textbf{Incorporating second-order information:} In the work of \cite{helmberg2014spectral}, the idea of incorporating second-order information 
	with low rank approximation (a version of block eigenvectors) has been explored and the algorithm, CB-diag, appears 
	to achieve $10^{-4}$ precision faster than the spectral bundle method with $r_c=1$ \cite[Section 7]{helmberg2014spectral}. Yet no convergence theory has been given for this method. Is it possible to adapt some of the proof techniques here to provide faster convergence guarantees for CB-diag?
\end{itemize}

	\section*{Acknowledgments.} We would like to thank insightful discussions with Michael L. Overton, Adrian Lewis, James Renegar, Yudong Chen, Madeleine Udell, and Zhenan Fan. We would also like to thank the editor and two anonymous referees for their constructive comments.

\bibliographystyle{alpha}
\bibliography{references}	
	\appendix
\section{A historical remark on $(\rpast,\rcurrent)$-SpecBM}\label{sec: historicalRemark} 
The algorithm presented in Helmberg and Rendle's paper \cite[Algorithm 4.1]{helmberg2000spectral} requires $\rcurrent =1$ and allows the parameter $\rpast $ to vary by the user. The requirement on $\rcurrent$ might be due to the fact that the authors want to avoid guessing the correct multiplicity of the optimal solution as done in previous works such as \cite{cullum1975minimization,polak1982nondifferentiable,overton1992large} and the use of entire spectrum as done in \cite{overton1992large,oustry2000second}, since requiring SpecBM with $\rcurrent =1$ is enough for their method to converge. Nevertheless, in their implementation, on  \cite[page 690]{helmberg2000spectral}, it is mentioned that ``$P^k$ may be enriched with additional Lanczos-vectors
from the eigenvalue computation''. In our notation, this means that we allow $\rcurrent>1$. This is made more clear in the book chapter \cite[page 330]{helmberg2000bundle}, ``...add $n_A$ Lanczos vectors corresponding to the largest eigenvalues of $T_i$''. In our notations, this means set $\rcurrent = n_A >1$. A more systematic approach to the spectral bundle method using past and current eigenvectors can also be found in \cite[Section 3.4.2 and Section 3.4.3]{lemarechal2000nonsmooth}, where multiple eigenvectors of the current $Z(z_t)$ are computed\footnote{Note the way of dealing with past eigenvectors is different from the approach used in this paper.}.

A subtle difference between the method in \cite{lemarechal2000nonsmooth,oustry2000second} and the one presented in this paper is that the algorithmic parameter (chosen at each iteration) in \cite{lemarechal2000nonsmooth,oustry2000second} is $\epsilon$ rather than $\rcurrent$. The quantity $\epsilon$ is a  quantity associated with the $\epsilon$-enlargement of the largest eigenvalue: given a symmetric matrix $A$, its $\epsilon$-enlargement is defined as
\[
\Lambda_\epsilon :\,=\{\lambda_i\mid \lambda_i(A)\geq \lambda_1(A)-\epsilon \}.
\]
Accordingly, the $\rcurrent$ in \cite[eq. (3.21)]{lemarechal2000nonsmooth} is defined to be  
\[
\text{the cardinality of } \Lambda_\epsilon.
\]
We do not take this approach as our motivating applications in Section \ref{sec: discussion On low-rankness} have natural upper bounds on the solution rank which can be used to set $\rcurrent$ (even without knowledge of such bounds, we still guarantee the method converges, albeit sublinearly, for any $\rcurrent\geq 1$). Regardless, considering adaptive choice in model construction is an important practical direction as the upper bound information may not be available.

One primary reason for the use of an $\epsilon$-enlargement in \cite{lemarechal2000nonsmooth,oustry2000second} is its connection to Markovian dual bundle methods \cite[Chap.  XIII]{hiriart2013convex}, which utilizes the so-called $\epsilon$-subdifferential \cite[Chap. XI]{hiriart2013convex}. The $\epsilon$-enlargement can be used as an inner approximation of the $\epsilon$-subdifferential utilizing the structure of the largest eigenvalue function
\cite[eq. (12)]{oustry2000second}. However, as discussed in \cite[Beginning of Sec. 3.3]{lemarechal2000nonsmooth}, determining a good rule for selecting $\epsilon$ is hard as it has a bivalent role. In the case of our results on linear speedups, we would require $\epsilon$ to be a half or a constant fraction of the eigengap $\delta$ for the negative dual optimal slack matrix $-\Amap^*\ysol +C$, defined in \eqref{eq: definitionOfdelta}, so that when the iterate is near the solution with distance comparable to the eigengap, we can identify the rank. However, determining the eigengap a priori is even harder than an upper estimate of the rank of the primal solution for the applications described in Section \ref{sec: discussion On low-rankness}.
\section{Projecting to a scaled $\mathcal{S}_t$}\label{sec: projecttoS_t}
Recall the spectral bundle method needs to solve the subproblem
\[
\min_{(\eta,S)\in\mathcal{S}_t} f_t(\eta,S),
\]
where 
\begin{equation} 
\begin{aligned} 
f_t(\eta,S) :\;=& \inprod{b}{y_t}  +\inprod{\eta \bar{X}_t 
+V_tSV_t^\top}{C-\Amap^*y_t} \\
&+\frac{1}{2\rho_t}\twonorm{b-\Amap
	\left( \eta \bar{X}_t+V_tSV_t^\top \right)}^2, \\ 
\mathcal{S}_t  := &\{{S\succeq 0,\;\eta \geq 0,\;\tr(S)+\alpha \eta \leq \alpha}\}.
\end{aligned} 
\end{equation}
After rescaling this constraint set, we may consider
the constraint set as 
\[ 
\tilde{\mathcal{S}}= \{S\in \symMat_+^k,\, \eta\geq 0,\, \tr(S)+\eta \leq 1\},
\]
and a new objective $\tilde{f}_t(\eta,S) = f(\eta,\alpha S)$.

Below we detail how to project any $(\eta_0,S_0)\in \RR\times \symMat^{\bar{r}}$ on to the set $\tilde{\mathcal{S}}$, yielding some $(\eta^\star,S^\star)$. This can be done by diagonalizing and projecting onto a simplex as follows:
\begin{enumerate}
	\item Compute the eigenvalue decomposition of $S_0 = V\Lambda_0 V^\top$,
	where $\Lambda_0\in \symMat^{\bar{r}}$ is a diagonal matrix with diagonal $\vec{\lambda}_0=( \lambda_1,\dots,\lambda_{\bar{r}})$.
	\item Compute $(\eta^\star, \vec{\lambda}^\star) = \arg\min_{\eta+\sum_{i=1}^{\bar{r}}\lambda_i \leq 1,\;\eta \geq 0,\;\lambda_i\geq 0}\twonorm{(\eta_0,\vec{\lambda}_0)-(\eta,\vec{\lambda})}$.
	\item Form $S^\star = V\diag(\vec{\lambda^\star})V^\top$.
	Here $\diag(\lambda)$ forms a diagonal matrix with the vector $\lambda$ on the diagonal.
\end{enumerate}
The main computational cost is the eigenvalue decomposition which requires $\bigO(\bar{r}^3)$ time. The second 
step requires projection to the convex hull of probability simplex and the origin, which can be done 
in $\bigO(\bar{r}\log \bar{r})$ time \cite{wang2013projection}. The correctness of this procedure
can be verified as in \cite[Lemma 3.1]{allen2017linear}
and \cite[Lemma 6]{garber2019convergence}.
\section{Relationship between 
Lemma \ref{lem: linearlyConvergent} and  \cite{drusvyatskiy2018error,drusvyatskiy2019efficiency}}\label{sec: relationshipDimaAdrian}
The two papers \cite{drusvyatskiy2018error,drusvyatskiy2019efficiency} study the prox-linear method, and the concept of a quadratically accurate model is not explicitly mentioned. However, a combination of proofs there can be employed to establish Lemma \ref{lem: linearlyConvergent} assuming a descent step is taken $y_{t+1} = z_{t+1}$ (which we handle in \eqref{eq: descentStepASBMLocalLinearConvergence}).

Specifically, based on a quadratically accurate model, we obtain 
inequalities  \eqref{eq: approximationModelStep1} (the second inequality) and \eqref{eq: approximationModelStep2} (with a transformation based on the law of cosine). And these two correspond to 
\cite[Eq. (3.2) and (3.3)]{drusvyatskiy2018error}. The authors of 
\cite{drusvyatskiy2018error} then prove the linear convergence based on an error bound condition \cite[Definition 3.1]{drusvyatskiy2018error} that is specific to the prox-linear method. To use the proof there for the bundle method, we define the error bound condition as  $\dist(y_t,\ysolset)\leq \beta_0 \twonorm{y_{t+1}-y_t}$ where $y_{t+1}= \arg\min_z \bar{F}_t(z) + \frac{\rho}{2}\twonorm{z-y_t}^2$ for some $\beta_0>0$. For objective with quadratic growth, the error bound condition for prox-linear method can be proved using \cite[Corollary 3.6]{drusvyatskiy2018error}
which is based on a step-lengths comparison inequality \cite[Inequality (3.11)]{drusvyatskiy2018error}. To use the proof in \cite[Corollary 3.6]{drusvyatskiy2018error}, we need to define an appropriate notion of  step-lengths comparison inequality and prove it holds under the quadratically accurate model. We define the comparison inequality in the context of bundle methods as 
$\twonorm{y_{t+1}-y_t}\geq \beta_1\twonorm{\hat{y}_t-y_t}$ where $\hat{y} = \arg\min_z F(z) + \frac{\rho}{2}\twonorm{z-y_t}^2$.
This inequality can be proved based on \cite[the proof of Theorem 4.5]{drusvyatskiy2019efficiency} which only uses that the given model $\bar{F}_t$ is a quadratically accurate model. 
\end{document}